\documentclass[pdflatex,sn-mathphys-num]{sn-jnl}% Math and Physical Sciences Numbered Reference Style
\usepackage{amsmath, amsthm, amssymb,  amsfonts, graphicx}
\usepackage{multirow}
\usepackage{booktabs}
\usepackage{easyReview}
\usepackage{enumitem}
\usepackage{caption}
\usepackage{subcaption}
\usepackage{array}
\usepackage{float} 
\usepackage{amscd}
\usepackage{enumitem} % Include this package for enhanced enumeration
\allowdisplaybreaks
\usepackage{hyperref}
\usepackage{bm}  % For bold Greek letters like \bm{\alpha}
\usepackage{upgreek}

\usepackage{mathrsfs}%
\usepackage[title]{appendix}%
\usepackage{xcolor}%
\usepackage{textcomp}%
\usepackage{manyfoot}%
\usepackage{booktabs}%
\usepackage{algcompatible}
\usepackage{algorithm}
\usepackage{algorithmicx}%
\usepackage{algpseudocode}%
\usepackage{listings}%
\usepackage{etoolbox}

\makeatletter
\patchcmd{\algorithmic}{\item[\ALG@line\hskip\algorithmicindent]}{\item[]}{}{}
\makeatother

\newcommand{\Idx}[1]{\mathfrak{I}_{#1}}          % Index set {1,...,#1}
\newcommand{\Rplus}[1]{\mathbb{R}_{\geq}^{#1}}  % Non-negative orthant
\newcommand{\xlm}{z^{\dagger}}                 % Optimal point
\newcommand{\obj}[1]{\mathcal{\theta}_{#1}}              % j-th objective function

\newcommand{\fvec}{\boldsymbol{\bm{\theta}}} % Vector-valued objective function
\newcommand{\Feas}{\Omega}                      % Feasible set
\newcommand{\mult}{\mathsf{\eta}}                 % Multiplier vector
\newcommand{\Kconst}{\mathcal{K}}                    % Constant K
\newcommand{\Lconst}{\ell}                      % Constant L
\newcommand{\muc}{\varsigma}                       % Constant M
\newcommand{\rhoc}{\vartheta} 
\newcommand{\gdf}{\mathcal{G}}

\theoremstyle{thmstyleone}%
\newtheorem{theorem}{Theorem}[section]% meant for sectionwise numbers
\theoremstyle{thmstyletwo}%
\newtheorem{remark}{Remark}[section]%
\newtheorem{assumption}{Assumption}
\newtheorem{lemma}{Lemma}[section]

\theoremstyle{thmstylethree}%
\newtheorem{definition}{Definition}%

\begin{document}

\title[Global Descent Method for Non-convex MOP]{Global Descent Method for Non-convex Multi-objective Optimization Problems}


% \author*[1,2]{\fnm{First} \sur{Author}}\email{iauthor@gmail.com}



% \affil*[1]{\orgdiv{Department}, \orgname{Organization}, \orgaddress{\street{Street}, \city{City}, \postcode{100190}, \state{State}, \country{Country}}}

% \affil[2]{\orgdiv{Department}, \orgname{Organization}, \orgaddress{\street{Street}, \city{City}, \postcode{10587}, \state{State}, \country{Country}}}

% \affil[3]{\orgdiv{Department}, \orgname{Organization}, \orgaddress{\street{Street}, \city{City}, \postcode{610101}, \state{State}, \country{Country}}}

\author[1]{\fnm{Bikram} \sur{Adhikary}}
\equalcont{These authors contributed equally to this work.}
\author*[1]{\fnm{Md Abu Talhamainuddin} \sur{Ansary}}\email{md.abutalha2009@gmail.com}
\equalcont{These authors contributed equally to this work.}
\author[2]{ \fnm{Savin} \sur{Trean{\c{t}}{\u{a}}}}
\affil[1]{\orgdiv{Department of Mathematics}, \orgname{Indian Institute of Technology Jodhpur} \city{Jodhpur}, \postcode{342030}, \state{Rajasthan}, \country{India}}
\affil[2]{Department Applied Mathematics, National University of Science and Technology Politehnica Bucharest, 060042 Bucharest, Romania}
\equalcont{These authors contributed equally to this work.}
%\affil[2]{\orgdiv{Department of Mathematics}, \orgname{Indian Institute of Technology Jodhpur}, \city{Jodhpur}, \postcode{342030}, \state{Rajasthan}, \country{India}}}

% \affil[3]{\orgdiv{Department}, \orgname{Organization}, \orgaddress{\street{Street}, \city{City}, \postcode{610101}, \state{State}, \country{Country}}}


\abstract{In this paper, we develop a global descent method for non-convex multi-objective optimization problems. The proposed approach builds upon foundational concepts from single-objective global descent techniques while removing the need for predefined scalars or ordering information of objective functions. Initially, the proposed method identifies a local weak efficient solution using any suitable descent algorithm, then applies an auxiliary function termed the multi-objective global descent function to systematically transition toward improved local weak efficient solutions. It is justified that this method can generate a global Pareto front for non-convex problems, which has many different local Pareto fronts. Finally, comprehensive numerical experiments on benchmark non-convex multi-objective optimization problems have been done to demonstrate the method’s robustness, scalability and effectiveness of the proposed method.}
\keywords{Global optimization, multi-objective optimization, global descent method, non-convex optimization, Pareto front}


\pacs[2020 MSC Classification]{90C29, 90C26, 65K10, 49M99}

\maketitle

\section{Introduction}\label{sec_intro}
%%%%%%%%%%%%%%%
% \cite{leschine1992interactive, fliege2001olaf}
%\cite{palermo2003system, tavana2004subjective}
%\cite{babaei2015multi}
%\cite{gravel1992multimanagement, de1992minimization}
%\cite{hutterer2003location}
Multi-objective optimization addresses problems involving multiple, often conflicting objectives that must be optimized simultaneously. A major challenge in such settings is the identification of Pareto-optimal solutions, where no objective can be improved without deteriorating another. It has widespread applications in various scientific and engineering domains, including environmental analysis, space exploration, portfolio management, management science, and medical science etc.. These real-world problems often involve complex trade-offs that cannot be adequately addressed by single-objective optimization, thereby increasing the need for sophisticated methods capable of effectively balancing multiple competing objectives in a principled and efficient manner.
\par Traditional approaches to solving multi-objective optimization problems typically fall into two categories: scalarization methods and heuristic approaches. Scalarization techniques \cite{ Deb2001-tt, Ehrgott2005, Miettinen2012-rv} reduce a multi-objective problem to a single-objective one using weighted combinations or constraint reformulations. However, they are sensitive to parameter selection and often fail to capture the full Pareto front, particularly in non-convex settings. On the other hand, heuristic approaches, such as evolutionary algorithms \cite{laumanns2002combining, mostaghim2007multi}, offer stochastic search mechanisms but lack of theoretical guarantee of convergence to the solution, making them computationally expensive and often inefficient for large-scale applications.
\par Recent advances in descent-based optimization have extended classical single-objective descent methods to multi-objective problems. These include methods for smooth unconstrained problems \cite{ ansary2015modified,  fliege2009newton, fliege2000steepest, ram1, ram2, qu2013trust}, constrained scenarios \cite{ ansary2019sequential, ansary2020sequential, ansary2021sqcqp, fliege2016sqp}, non-smooth optimization \cite{ansary2023proximal, bento2014proximal, tanabe2019proximal}, and uncertain multi-objective optimization problems with finite uncertainty \cite{kumar1,kumar2,kumar3}. While promising, these techniques frequently converge to local Pareto-optimal solutions and remain limited in their ability to handle non-convex problems, which is a central challenge in many real-world applications. This motivates the development of more globally oriented descent strategies, such as the method proposed in this work, to effectively explore and approximate the global Pareto front in non-convex problems.
\par Non-convex multi-objective optimization problems introduce additional \linebreak
challenges due to the complex non-convex nature of the objective functions and intricate trade-offs among competing objectives, where multiple local Pareto-optimal solutions complicate the search for global Pareto-optimal solutions. Traditional gradient-based techniques suffer from premature convergence to local solutions, while metaheuristic methods such as genetic algorithms and particle swarm optimization require excessive computational resources and lack strong theoretical foundations. These limitations highlight the need for a robust optimization framework that can efficiently explore non-convex Pareto-optimal landscapes while still being supported by solid theoretical foundations.
\par In the context of the single-objective optimization problem, the filled function method introduced by Renpu \cite{renpu1990filled} has emerged as a prominent strategy among \linebreak
function-modification techniques for escaping local minima. Despite its theoretical appeal, early implementations of the method encountered significant numerical challenges. To address these issues, refined versions of filled functions were developed, notably by Zhang et al. \cite{zhang2004new}, which improved numerical stability and convergence behavior. Extending this line of work, Ng et al. \cite{duanli2010global} proposed a generalized framework involving auxiliary functions, known as global descent functions. These functions inherit the favorable analytical properties of filled functions while ensuring the existence of a descent direction in the neighborhood of a superior solution, thereby enhancing their applicability in complex non-convex optimization scenarios.
\par Building upon the foundational theory of global descent functions proposed by  Ng et al. \cite{duanli2010global}, this work extends the auxiliary-function-based strategy to the multi-objective optimization setting. Multi-objective problems introduce unique challenges, including the lack of a total ordering of objective vectors and the presence of locally Pareto-optimal solutions that can trap descent-based methods. Unlike scalar problems, where a single objective provides a clear descent direction, multi-objective landscapes often feature conflicting gradients and disconnected efficient regions. Our proposed method addresses these difficulties by constructing a multi-objective global descent function that ensures descent across multiple objectives and enables escape from locally nondominated regions toward globally Pareto-optimal solutions. This framework facilitates deeper and more reliable exploration of non-convex Pareto fronts, while maintaining descent consistency grounded in solid theoretical principles.
\par The principal contributions of this work are as follows:
\begin{itemize}
\item introduction of a {multi-objective auxiliary function} (Definition~\ref{def_mogdf}) that enhances solution exploration in complex non-convex optimization problems;
    \item formulation of a global descent-based optimization algorithm that overcomes the shortcomings of conventional scalarization and heuristic approaches;
    \item development of a multi-objective global descent method to facilitate effective transitions from local to global Pareto-optimal solutions;   
%%%%
    \item theoretical guarantee of descent toward globally Pareto-optimal solutions, including the existence of descent directions near superior points (Theorems \ref{theorem_c1}, \ref{theorem_3.5}, and \ref{theorem_3.6});   
    \item empirical validation on diverse benchmark problems, demonstrating the proposed method’s ability to escape local Pareto traps and improve solution quality across conflicting objectives.
\end{itemize}
The novelty of this study lies in the structured integration of auxiliary functions into a rigorous multi-objective optimization framework. Unlike conventional scalarization-based approaches, which require predefined weight factors, or metaheuristic techniques, which lack convergence guarantees, our approach dynamically refines solutions to achieve global efficiency. The proposed method is particularly advantageous for non-convex problems, ensuring the discovery of high-quality Pareto fronts that traditional techniques fail to achieve. By introducing a global descent mechanism, this work provides a significant advancement in multi-objective optimization methodologies, enhancing the ability to tackle complex, real-world optimization challenges across multiple domains. The proposed framework not only ensures theoretical rigor but also offers a scalable and computationally efficient approach to solving high-dimensional multi-objective problems.
\par The remainder of this paper is structured as follows. Section \ref{sec_pre} presents the necessary preliminaries and theoretical foundations related to the proposed approach. In Section \ref{MOGDM}, the multi-objective global descent method is introduced, incorporating a novel multi-objective auxiliary function to improve global optimization performance. An algorithm is proposed in this section. Section \ref{sec_exe} presents extensive numerical experiments conducted to validate the effectiveness of the proposed method. Finally, concluding remarks and potential directions for future research are discussed in the final section.
\section{Preliminaries}\label{sec_pre}
Throughout the paper, for any $p\in \mathbb{N}$, we use the following notations:
\begin{eqnarray*}
    \Idx{p} &:=& \{1, 2, \dots, p\};\\
    \Rplus{p}&:=& \{z \in \mathbb{R}^p \mid z_i \geq 0, \forall i \in \Idx{p}\}.\\
    % \Rpp{p} &:=& int(\Rplus{p}).
\end{eqnarray*}
For $x,y \in \mathbb{R}^p$, vector inequalities are interpreted component-wise. Let us consider an unconstrained multi-objective optimization problem:
\begin{equation*}\label{eq2.1}
(MOP): \underset{z \in \Feas \subseteq \mathbb{R}^n}{\min} \fvec(z) = \left( \obj{1}(z), \obj{2}(z), \dots, \obj{m}(z) \right)
\end{equation*}
where \( m \geq 2 \) and \( \obj{j}: \mathbb{R}^{n} \to \mathbb{R},~ j \in \Idx{m} \) are continuously differentiable functions.
The gradient of \( \obj{j} \) at a point \( z \) is denoted by \( \nabla \obj{j}(z) \).
\par If there exists a point \( \xlm \in \Feas \) at which all objective functions attain their minimum values simultaneously, then \( \xlm \) is referred to as an ideal solution. However, in practice, improving the value of one objective function often results in the deterioration of another. Consequently, the notion of efficiency, rather than optimality, becomes the central focus in multi-objective optimization. A feasible point \( \xlm \) is said to be a global efficient solution of $(\mathrm{MOP})$ if there exists no other \( z \in \Feas \) such that \( \fvec(z) \leq \fvec(\xlm) \) and \( \fvec(z) \neq \fvec(\xlm) \) hold where $\fvec(z) = \left( \obj{1}(z), \obj{2}(z), \dots, \obj{m}(z) \right).$ Moreover, a feasible point \( \xlm \) is termed a global weak efficient solution of $(\mathrm{MOP})$ if there exists no \( z \in \mathbb{R}^n \) such that \( \fvec(z) < \fvec(\xlm) \) holds. A feasible point \( \xlm \) is said to be a local efficient solution of $(\mathrm{MOP})$ if there exists a neighborhood of  \( \xlm \) ($N(\xlm,r)$, $r>0$) such that there does not exist any $z\in N(\xlm,r)$ such that  \( \fvec(z) \leq \fvec(\xlm) \) and \( \fvec(z) \neq \fvec(\xlm) \) hold. Similarly, \( \xlm \) is a local weak efficient solution of  $(\mathrm{MOP})$ if does not exist any $z\in N(\xlm,r)$ such that  \( \fvec(z) <\fvec(\xlm) \) holds. If $\fvec_j$ is convex for every $j$, then a local (weak) efficient solution will be a global (weak) efficient solution of  $(\mathrm{MOP})$. If \( X^\dagger \) denotes the set of all (global/local) efficient solutions to $(\mathrm{MOP})$, then the corresponding image \( \fvec(X^\dagger) \) is known as the (global/local) Pareto front of the problem. Convex multi-objective optimization problems have one global Pareto front, but non-convex multi-objective optimization problems may have different local Pareto fronts.
\par The theorem stated below provides a necessary condition for a point \( \xlm \) to qualify as a local weak efficient solution of $(\mathrm{MOP})$, assuming that the objective functions \( \obj{k}: \mathbb{R}^n \rightarrow \mathbb{R} \) are continuously differentiable for all \( k \in \Idx{m} \).
\begin{theorem}[{\bf First order necessary condition for weak efficiency}]\label{theorem1.1}
If any feasible point $\xlm$ is a weak efficient solution of $(\mathrm{MOP})$, then there exists $\mult \in \Rplus{m}$, $\mult \neq \mathbf{0}^{m}$, satisfying
\begin{flalign}
 \sum_{k \in \Idx{m}} \mult_k \nabla \obj{k}(\xlm) = 0. \label{eq1.2}
\end{flalign}
\end{theorem}
A feasible point \(\xlm\) is said to be a critical point of $(\mathrm{MOP})$ if  (\ref{eq1.2})  holds at $\xlm$ with some $\mult\in \Rplus{m}\setminus\{0^m\}$. A critical point is a weak efficient solution if $\fvec_j$ is a convex function for all $j$. In case of strictly convex functions $\fvec_j$, a critical point is an efficient solution.
%The set of the vectors $\mult \in \Rplus{m} \backslash \{\mathbf{0}^{m}\} $, satisfying (\ref{eq1.2}) are referred to as Fritz John multipliers corresponding to $\xlm$.
% However, relying solely on the Fritz John condition does not guarantee the existence of $\mult_k > 0$ for at least one $k \in \Idx{m}$. \alert{Hence, specific constraint qualifications or regularity conditions must be met to ensure this occurrence. In our analysis, we will consistently consider the Mangasarian–Fromovitz constraint qualification (MFCQ) in \cite{mangasarian1967fritz}.}

% \begin{definition}
% The MFCQ is considered to hold at \(x \in \Feas\),  if there exists a \(z \in \mathbb{R}^n\) such that \(\nabla g_i(z)^T z < 0\) for \(i \in I(z)= \{i \in \Idx{p} : g_i(z) = 0\} \), the set of active constraints.
% \end{definition} 
% By Gordan's theorem of alternative, the MFCQ holds at a feasible point $x$ if and only if  \(\sum_{i \in I(z)} \mu_i \nabla g_i(z) = 0\), with \(\mu_i \geq 0\), has no non-zero solution (see \cite{ansary2021sqcqp}).\\
%%%%%%%%%%%%%%%%%%%%%%%%%%%%%%%%%%%%%%%%%%%%%%%%%%%%%%%%%%%%%%%%%%%%%%%
%\section{Multi-objective global descent method}\label{MOGDM}
\section{Multi-objective global descent method}\label{MOGDM}
This section begins by outlining the core assumptions that underpin the proposed framework. A general definition of the multi-objective global descent function is then introduced. Building upon this, we construct a two-parameter family of global descent functions for multi-objective optimization problems $(MOPs)$. It is subsequently demonstrated that this class of functions satisfies the necessary conditions for global descent. The stated assumptions remain in effect throughout the remainder of the paper.
\begin{assumption}\label{assumption1}
The level set \( \Feas = \left\{ z \in \mathbb{R}^n : \fvec(z) \leq \fvec(z^0) \right\} \), where \( z^0 \) denotes the initial approximation, is assumed to be a compact and connected subset of \( \mathbb{R}^n \) with nonempty interior.
\end{assumption}
\begin{assumption}\label{assumption2}
The objective functions \( \obj{j}:\Feas \to \mathbb{R}^m \) are continuously differentiable and satisfy the Lipschitz continuity condition. Specifically, for any \( z^{(1)}, z^{(2)} \in \Feas \)
\begin{equation} \label{eq1.3}
    |\obj{j}(z^{(1)}) - \obj{j}(z^{(2)})| \leq \Lconst_j \|z^{(1)} - z^{(2)}\|, ~ j \in \Idx{m},
\end{equation}
where \( 0<\Lconst_j < \infty \) are Lipschitz constants and \( \| \cdot \| \) represents the Euclidean norm. For $\Lconst= \max\left\{\Lconst_1,\Lconst_2,\dots, \Lconst_m\right\}$,
\begin{equation} \label{eq1.4}
    |\obj{j}(z^{(1)}) - \obj{j}(z^{(2)})| \leq \Lconst \|z^{(1)} - z^{(2)}\|, ~ j \in \Idx{m}. 
\end{equation}
\end{assumption}
\begin{assumption}\label{assumption3}
The set of all values corresponding to the local weak efficient solutions 
\begin{equation} \label{eq1.5}
\fvec^\dagger_{\text{all}} = \{\fvec(\xlm): \xlm \text{ is a local weak efficient solution of } (MOP) \text{ over } \Feas \}.
\end{equation}
is finite.
\end{assumption}
\begin{assumption}\label{assumption4}
The set of global efficient solutions of $(\mathrm{MOP})$ lies in the interior of $\Feas$.
\end{assumption}
% \subsection*{Implications of the assumptions}
\paragraph{Implications of the assumptions}
\begin{itemize}
    \item From Assumption \ref{assumption1}, there exists a constant \( \Kconst > 0 \) such that the maximum distance between any two points in $\Feas$ is bounded,
        \begin{equation}\label{eq1.7}
        0 < \max_{z^{(1)}, z^{(2)} \in \Feas} \|z^{(1)} - z^{(2)}\| \leq \Kconst.
    \end{equation}
    
    \item Assumption \ref{assumption2} implies the existence of a constant \( \mathcal{M}> 0 \) such that
    \begin{equation}\label{eq1.8}
        0 < \max_{x \in \Feas} \|\nabla \obj{j}(z)\| \leq \mathcal{M}.        
    \end{equation}
    % \alert{correct all of places where $\mathcal{M}$ is used}
    
    \item Assumption \ref{assumption3} ensures that while the number of local weak efficient solutions in $\Feas$ could be infinite, the corresponding set of function values is finite.
    % \item \alert{Assumption \ref{assumption4} guarantees that all global weak efficient solutions are located strictly within the interior of \( \Feas \), facilitating gradient-based descent methods.}
\end{itemize}
We now proceed to define the class of multi-objective global descent functions. For this purpose, let \( \xlm \in \operatorname{int}(\Feas) \) be a known local weak efficient solution of $(\mathrm{MOP})$. We introduce the following subset of the feasible region:
\begin{equation}\label{eq_mop_basin}
  \hat{\Feas}(\xlm) = \left\{ z \in \mathrm{int}(\Feas) : z \neq \xlm,\ \obj{j}(z) \geq \obj{j}(\xlm) \text{ for at least one } j \in \Idx{m} \right\}.  
\end{equation}
\begin{definition}[Multi-objective global descent function]\label{def_mogdf}
    A function \( \gdf_{\xlm} : \Feas \to \mathbb{R}^m \) is said to be a multi-objective global descent function of $\fvec$ at \( \xlm \) if it satisfies the following conditions:
\begin{enumerate}[label=(C\arabic*)]
% \begin{itemize}
    \item \( \xlm \) is a local weak efficient solution of \( -\gdf_{\xlm} \) over \( \Feas \);
    \item \( \gdf_{\xlm} \) has no Fritz John points in the set \(  \hat{\Feas}(\xlm) \);
    \item if \( z^{\dagger \dagger} \in \mathrm{int}(\Feas) \) is a local weak efficient solution of \( (MOP) \) with \( \obj{j}(z^{\dagger \dagger}) < \obj{j}(\xlm) \) for all \( j \in \Idx{m} \), then \( \gdf_{\xlm} \) has a local weak efficient solution \( z' \) such that \( z' \in N(z^{\dagger \dagger},\epsilon) \subset \Feas \) and \( \obj{j}(z) < \obj{j}(\xlm) \) for all \( x \in N(z^{\dagger \dagger},\epsilon) \), where \( N(z^{\dagger \dagger},\epsilon) \) denotes an \( \epsilon \)-neighborhood of \( z^{\dagger \dagger} \).
\end{enumerate}
\end{definition}
We propose a family of two-parameter multi-objective global descent functions for \( (MOP) \) at a given local weak efficient solution \( \xlm \) over \( \Feas \). Define
\begin{equation}\label{eq8}
  \gdf_{j,\muc,\rhoc,\xlm}(z) = A_\muc(\obj{j}(z) - \obj{j}(\xlm)) - \rhoc \| z - \xlm \|,  
\end{equation}
where \( \rhoc > 0 \), \( 0 < \muc < 1 \), and
\begin{equation*}
A_\muc(y) = y \cdot V_\muc(y),
\end{equation*}
with \( V_\muc : \mathbb{R} \to \mathbb{R} \) continuously differentiable and satisfying:
\begin{enumerate}[label=(V\arabic*)]
    \item \( V_\muc(-\tau) = 1 \), \( V_\muc(0) = \muc \), and \( V_\muc(y) \geq c \muc \) for all \( y \);
    \item \( V_\muc'(y) < 0 \) for all \( y < 0 \), and \( -c' \muc \leq V_\muc'(y) \leq 0 \) for all \( y \geq 0 \),
\end{enumerate}
where \( V_\muc' \) is the derivative of \( V_\muc \), \( \tau > 0 \) is small, \( 0 < c \leq 1 \), and \( c' \geq 0 \) satisfies \( \lim_{\muc \to 0} \muc c' = 0 \).

% \paragraph{Parameter Setting.}
The parameter \(\tau\) must satisfy:
\begin{equation}\label{eq2.2}
    0 < \tau < \min \left\{ |\fvec^{\dagger} - \fvec^{\dagger \dagger}| : \fvec^{\dagger}, \fvec^{\dagger \dagger} \in {\fvec^{\dagger}}_{\text{all}}, \fvec^{\dagger} \neq \fvec^{\dagger \dagger} \right\},
\end{equation}
where \( {\fvec^{\dagger}}_{\text{all}} \) is defined in Assumption~\ref{assumption3}. For all such \( \fvec^\dagger, \fvec^{\dagger \dagger} \), the inequality \linebreak
\( \fvec^{\dagger \dagger} < \fvec^\dagger - \tau \) holds. In practice, the descent algorithm described in subsection~\ref{sec_alg} is insensitive to the value of \(\tau\), and hence it is set to \( \tau = 1 \) in computations.

% \paragraph{Purpose.}
Our goal is to transition from one local weak efficient solution of \( (MOP) \) to a better one at each iteration using this auxiliary function, termed as {multi-objective global descent function}.
% \subsection{Basic Properties of \( A_\muc \) and \( V_\muc \)}
%%%%%%%%%%%%%%%%%%%%%%%%%%%%%%%%%%%%%%%%%%%%%%%%%%%%%%%%%%%%%%%%%
% \newpage
\par The following lemma, as stated in~\cite{duanli2010global}, is included here for completeness and will be instrumental in the subsequent analysis.
% The following basic properties will be useful in later analysis.
\begin{lemma}\label{lemma3.1}~
\begin{enumerate}
    \item \( V_\muc(y) > 1 \) for all \( y < -\tau \); \label{property1_lemma3.1}
    \item \( V_\muc(y) \leq \muc \) for all \( y \geq 0 \); in particular, if \( c' = 0 \), then \( V_\muc(y) = \muc \) for all \( y \geq 0 \);\label{property2_lemma3.1}
    \item If \( y^{(1)} < y^{(2)} \leq -\tau < 0 \), then \( 0 < y^{(2)} - y^{(1)} < A_\muc(y^{(2)}) - A_\muc(y^{(1)}) \); \label{property3_lemma3.1}
    \item \( A_\muc'(0) = \muc \), and \( A_\muc'(y) \leq \muc \) for all \( y > 0 \); if \( c' = 0 \), then \( A_\muc'(y) = \muc \) for all \( y > 0 \). \label{property4_lemma3.1}
\end{enumerate}
\end{lemma}
%\noindent\textbf{Remark.}
\begin{remark}
In the multi-objective setting, we introduce a novel extension of the set $\hat{X}$ originally defined in Section~2 of~\cite{duanli2010global}, and denote it by \( \hat{\Feas} \) (see~\eqref{eq_mop_basin}). This newly defined set plays a central role in characterizing the behavior of the proposed multi-objective global descent function \( \gdf_{j,\muc,\rhoc,\xlm} \), as formulated in Equation~(\ref{eq8}). However, establishing that \( \gdf_{j,\muc,\rhoc,\xlm} \) satisfies conditions (C1), (C2) and (C3) is nontrivial and introduces substantial theoretical challenges. Notably, the validity of these conditions hinges on the careful selection of the parameters \( \muc \) and \( \rhoc \). In the following theoretical developments (Theorems \ref{theorem_c1}, \ref{theorem_3.5} and \ref{theorem_3.6}), we formally investigate the conditions under which suitable choices of the parameters \( \muc \) and \( \rhoc \) ensure that \( \gdf_{j,\muc,\rhoc,\xlm} \) qualifies as a valid multi-objective global descent function.
\end{remark}
%\subsection{Satisfaction of condition (C1)}
\begin{theorem}\label{Lemma3.2}
Let \( \bar{z} \in  \hat{\Feas}(\xlm) \). If \( \rhoc > 0 \) and \( 0 < \muc < \min\left\{1, \frac{\rhoc}{\Lconst} \right\} \), then
$$- \gdf_{j, \muc, \rhoc, \xlm}(\bar{z}) > - \gdf_{j, \muc, \rhoc, \xlm}(\xlm)$$ holds for at least one $j$.
\end{theorem}
\begin{proof}
Let \( \bar{z} \in  \hat{\Feas}(\xlm) \), so that \( \obj{\bar{j}}(\bar{z}) \geq \obj{\bar{j}}(\xlm) \) for at least one \( \bar{j} \in \Idx{m} \). We know from Inequality~\ref{eq1.4} of Assumption~\ref{assumption2}, that
\[
0 < \obj{\bar{j}}(\bar{z}) - \obj{\bar{j}}(\xlm) \leq \Lconst \| \bar{z} - \xlm \|.
\]
By Property~\ref{property2_lemma3.1} of Lemma~\ref{lemma3.1}, we have
\[
V_{\muc}(\obj{\bar{j}}(\bar{z}) - \obj{\bar{j}}(\xlm)) \leq \muc.
\]
Thus,
\[
A_{\muc}(\obj{\bar{j}}(\bar{z}) - \obj{\bar{j}}(\xlm)) = [\obj{\bar{j}}(\bar{z}) - \obj{\bar{j}}(\xlm)] \cdot V_{\muc}(\obj{\bar{j}}(\bar{z}) - \obj{\bar{j}}(\xlm)) \leq \muc \Lconst \| \bar{z} - \xlm \|.
\]

Since \( \bar{z} \in \hat{\Feas}(\xlm) \) implies \( \| \bar{z} - \xlm \| > 0 \), and if \( \rhoc > 0 \), \( 0 < \muc < \min\left\{ 1, \frac{\rhoc}{\Lconst} \right\} \), then
\[
\gdf_{\bar{j}, \muc, \rhoc, \xlm}(\bar{z}) \leq \muc \Lconst \| \bar{z} - \xlm \| - \rhoc \| \bar{z} - \xlm \| < 0 = \gdf_{\bar{j}, \muc, \rhoc, \xlm}(\xlm),
\]
which implies \( -\gdf_{j, \muc, \rhoc, \xlm}(\bar{z}) > -\gdf_{j, \muc, \rhoc, \xlm}(\xlm) \) for at least one \( j \).
\end{proof}
\begin{theorem}\label{theorem_c1}
If \( \rhoc > 0 \) and \( 0 < \muc < \min\left\{ 1, \frac{\rhoc}{\Lconst} \right\} \), then \( \xlm \) is a strict local weak efficient solution of \( -\gdf_{\muc,\rhoc,\xlm} \). If, in addition, \( \xlm \) be a global weak efficient solution of \( \fvec \) over \( \Feas \), then \( \gdf_{j, \muc, \rhoc, \xlm}(\xlm) < 0 \) for at least one \( j \).
\end{theorem}
\begin{proof}
Assume \( \xlm \in \mathrm{int}(\Feas) \) is a strict local weak efficient solution of \( (MOP) \). Then there exists an \( \epsilon \)-neighborhood \( N(\xlm,\epsilon) \subset \Feas \) such that no \( x \in N(\xlm,\epsilon) \) satisfies \( \fvec(z) < \fvec(\xlm) \). This implies \( \obj{j}(z) \geq \obj{j}(\xlm) \) for at least one \( j \), for all \( x \in N(\xlm,\epsilon) \). By Theorem~\ref{Lemma3.2}, if \( \rhoc > 0 \) and \( 0 < \muc < \min\left\{ 1, \frac{\rhoc}{\Lconst} \right\} \), then
\[
- \gdf_{j, \muc, \rhoc, \xlm}(z) > - \gdf_{j, \muc, \rhoc, \xlm}(\xlm), \quad \text{for at least one } j, \quad \forall x \in N(\xlm,\epsilon) \setminus \{\xlm\}.
\]
Thus, \( \xlm \) is a strict local weak efficient solution of \( -\gdf_{\muc,\rhoc,\xlm} \).

If \( \xlm \) is a global weak efficient solution of \( \fvec \) over \( \Feas \), then \( \obj{j}(z) \geq \obj{j}(\xlm) \) holds for at least one \( j \) and for all \( x \in \Feas \). The result then follows from Theorem~\ref{Lemma3.2}.
\end{proof}
%%%%%%%%%%%%%%%%%%%%%%%%%%%%%%%%%%%%%%%%%%%%%%%%%%%%%
%\subsection{Satisfaction of condition (C2)}
%%%%%%%%%%%%%%%%%%%%%%%%%%%%%%%%%%%%%%%%%%%%%%%%%%%%%%%%%%%
\begin{theorem}\label{Lemma3.3}
Let \( \mathcal{I}_1 \), \( \mathcal{I}_2 \), and \( \mathcal{I}_3 \) be three partitions of \( \Idx{m} \) such that
\begin{enumerate}
    \item For all \( j \in \mathcal{I}_1 \), \( \obj{j}(z) = \obj{j}(\xlm) \);
    \item For all \( j \in \mathcal{I}_2 \), \( \obj{j}(z) > \obj{j}(\xlm) \);
    \item For all \( j \in \mathcal{I}_3 \), \( \obj{j}(z) < \obj{j}(\xlm) \).
\end{enumerate}
Consider \( \bar{z} \in \operatorname{int}(\Feas) \setminus \xlm \), such that there exists \( \mult \in \Rplus{m} \setminus \{0^m\} \) satisfying
\[
\sum_{i=1}^{m} \mult_i \nabla \obj{i}(\bar{z}) = 0
\]
and \( \obj{j}(\bar{z}) \geq \obj{j}(\xlm) \) holds for at least one \( j \). Let \( s \) be a feasible direction satisfying \( s^T (\bar{z} - \xlm) > 0 \), and for the partitions \( \mathcal{P}_1, \mathcal{P}_2, \mathcal{P}_3 \) of \( \mathcal{I}_1 \), with partitions \( \mathcal{Q}_1, \mathcal{Q}_2, \mathcal{Q}_3 \) of \( \mathcal{I}_2 \) and partitions \( \mathcal{R}_1, \mathcal{R}_2, \mathcal{R}_3 \) of \( \mathcal{I}_3 \), the following conditions hold:

%%%%
\begin{enumerate}[label=1\alph*.]
    \item \( s^T \nabla \obj{i}(\bar{z}) = 0 \), for all \( i \in \mathcal{P}_1 \);\label{case_1a}
    \item \( s^T \nabla \obj{i}(\bar{z}) > 0 \), for all \( i \in \mathcal{P}_2 \);\label{case_1b}
    \item \( s^T \nabla \obj{i}(\bar{z}) < 0 \), for all \( i \in \mathcal{P}_3 \);\label{case_1c}
\end{enumerate}
\begin{enumerate}[label=2\alph*.]
    \item \( s^T \nabla \obj{i}(\bar{z}) = 0 \), for all \( i \in \mathcal{Q}_1 \);\label{case_2a}
    \item \( s^T \nabla \obj{i}(\bar{z}) > 0 \), for all \( i \in \mathcal{Q}_2 \);\label{case_2b}
    \item \( s^T \nabla \obj{i}(\bar{z}) < 0 \), for all \( i \in \mathcal{Q}_3 \);\label{case_2c}
\end{enumerate}
\begin{enumerate}[label=3\alph*.]
    \item \( s^T \nabla \obj{i}(\bar{z}) = 0 \), for all \( i \in \mathcal{R}_1 \);\label{case_3a}
    \item \( s^T \nabla \obj{i}(\bar{z}) > 0 \), for all \( i \in \mathcal{R}_2 \);\label{case_3b}
    \item \( s^T \nabla \obj{i}(\bar{z}) < 0 \), for all \( i \in \mathcal{R}_3 \).\label{case_3c}
\end{enumerate}
    
%%%%%%%%%%%%%%%%%%%%%%%%%%
% For 
% \[
% \rhoc > \max \left\{ 0,\frac{\mathcal{C} \cdot \max_{i} \left(s^T \nabla \obj{i}(\bar{z})\right) \cdot \|\bar{z} - \xlm\|}{s^T (\bar{z} - \xlm)} \right\}
% \]
% and for \[
% 0 < \muc < \min \left\{
%     1, \,
%     \frac{\rhoc}{\mathcal{M}}, \,
%     \frac{\rhoc}{\|\bar{z} - \xlm\|} \cdot \frac{s^T (\bar{z} - \xlm)}{\max_i \left( s^T \nabla \obj{i}(\bar{z}) \right)}, \,
%     \frac{\rhoc}{c' \Lconst \|\bar{z} - \xlm\|^2} \cdot \frac{s^T (\bar{z} - \xlm)}{\max_i \left( - s^T \nabla \obj{i}(\bar{z}) \right)}
% \right\},
% \]
% the direction \( s \) constitutes a descent direction of \( \gdf_{\muc, \rhoc, \xlm} \) at \( \bar{z} \), where \( \mathcal{C} = A'_{\muc} \left( \obj{i}(\bar{z}) - \obj{i}(\xlm) \right) \).\\

%%%%%%%%%%%%%%%%%%%%%%%%%%%
% For $\rhoc > \max \left\{ 0,\frac{\mathcal{C} \cdot \max_{i \in \mathcal{Q}_2} \left(s^T \nabla \obj{i}(\bar{z})\right) \cdot \|\bar{z} - \xlm\|}{s^T (\bar{z} - \xlm)}
% \right\}$ and $0 < \muc < \min \left\{1,\frac{\rhoc}{\mathcal{M}}, \frac{\rhoc}{\|\bar{z} - \xlm\|} \cdot \frac{s^T (\bar{z} - \xlm)}{\max_i \left(s^T \nabla \obj{i}(\bar{z})\right)} \right\}$ for $c'=0$ or $0 < \muc < \min \left\{1,\frac{\rhoc}{\mathcal{M}}, \frac{\rhoc}{c' \Lconst \|\bar{z} - \xlm\|^2} \cdot \frac{s^T (\bar{z} - \xlm)}{\max_i \left(- s^T \nabla \obj{i}(\bar{z})\right)} \right\}$ for $c'> 0$, the direction \( s \) constitutes a descent direction of \( \gdf_{\muc, \rhoc, \xlm} \) at $\bar{z}$.
%%%%%%%%%%%%%%%%%%%%%%%%%%%%%%%%%
For
\begin{equation*}
\begin{aligned}
\rhoc > \max \Bigg\{ 0,\ 
& \frac{
    \mathcal{C} \cdot 
    \max_{j \in \mathcal{I}_3} \left(s^T \nabla \obj{j}(\bar{z})\right) \cdot 
    \|\bar{z} - \xlm\|
}{
    s^T (\bar{z} - \xlm)
}
\Bigg\} \quad \text{and}
\\
0 < \muc < \min \Bigg\{ 1,\ 
& \frac{\rhoc}{\|\bar{z} - \xlm\|} \cdot 
  \frac{s^T (\bar{z} - \xlm)}{
    \max_{i \in \mathcal{P}_2} \left(s^T \nabla \obj{i}(\bar{z})\right)}, \\
& \frac{\rhoc}{\|\bar{z} - \xlm\|} \cdot 
  \frac{s^T (\bar{z} - \xlm)}{
    \max_{i \in \mathcal{Q}_2} \left(s^T \nabla \obj{i}(\bar{z})\right)} 
\Bigg\} \quad \text{for } c' = 0, \\
\text{or~~} 0 < \muc < \min \Bigg\{ 1,\ 
& \frac{\rhoc}{\|\bar{z} - \xlm\|} \cdot 
  \frac{s^T (\bar{z} - \xlm)}{
    \max_{i \in \mathcal{P}_2} \left(s^T \nabla \obj{i}(\bar{z})\right)}, \\
& \frac{\rhoc}{\|\bar{z} - \xlm\|} \cdot 
  \frac{s^T (\bar{z} - \xlm)}{
    \max_{i \in \mathcal{Q}_2} \left(s^T \nabla \obj{i}(\bar{z})\right)}, \\
& \frac{\rhoc}{c' \Lconst \|\bar{z} - \xlm\|^2} \cdot 
  \frac{s^T (\bar{z} - \xlm)}{
    \max_{i \in \mathcal{Q}_3} \left( - s^T \nabla \obj{i}(\bar{z}) \right)} 
\Bigg\} \quad \text{for } c' > 0,
\end{aligned}
\end{equation*}
the direction \( s \) constitutes a descent direction of \( \gdf_{\muc, \rhoc, \xlm} \) at \( \bar{z} \).
\end{theorem}
\begin{proof}
% \alert{the main idea of the proof is that we first consider 3 different partitions of $\Delta_m,$ and for each partitions we need to find feasble directiojn d such thst $s^T...>0$ .Furthermore, each partitins require case by case analysis.}\\
% The central idea of the proof involves constructing three distinct partitions of the index set $\Idx{m}$. For each partition, we identify a feasible direction \( d \) such that \( s^T(\bar{z} - \xlm) > 0 \) and show d is descent direction under certain values of $\rhoc$ and $\muc$. The analysis then proceeds on a case-by-case basis, tailored to the structural properties of each partition.
The central idea of the proof involves constructing three distinct partitions of the index set \( \Idx{m} \). For each partition, we identify a feasible direction \( d \) such that \( s^T(\bar{z} - \xlm) > 0 \), and demonstrate that \( d \) serves as a descent direction under suitable values of \( \rhoc \) and \( \muc \). The analysis proceeds on a case-by-case basis, adapted to the structural properties of each partition.\\
Since \( \bar{z} \) is a critical point of \( \fvec \), we have
\begin{align*}
  \sum_{j=1}^m \mult_j \nabla \obj{j}(\bar{z}) = 0, \quad \mult \in \Rplus{m}, \mult \neq \mathbf{0}^{m}.
\end{align*}
Then $|\mathcal{P}_3|$, $|\mathcal{Q}_3|$ and $|\mathcal{R}_3|$ are always less than $|\Idx{m}|$; if not, it would lead to a contradiction that \( \bar{z} \) is a critical point of \( \fvec \).
\par Consider {\bm{$j \in \mathcal{I}_1$}}. Then for all \( j \in \mathcal{I}_1 \), the equality \( \obj{j}(z) = \obj{j}(\xlm) \) holds and it follows from Lemma~\ref{lemma3.1} that
\[
A'_{\muc} \left( \obj{i}(\bar{z}) - \obj{i}(\xlm) \right) = A'_{\muc}(0) = \muc.
\]
Accordingly, the following cases arise:
\begin{enumerate}[label=\textbf{Case 1\alph*.}]
    \item If \( i \in \mathcal{P}_1 \), then \( s^T \nabla \obj{i}(\bar{z}) = 0 \). For \( 0 < \muc < 1 \) and \( \rhoc > 0 \), it follows that:
    \begin{equation*}
        s^T \nabla \gdf_{i,\muc,\rhoc,\xlm}(\bar{z}) = - \rhoc \frac{s^T (\bar{z} - \xlm)}{\|\bar{z} - \xlm\|} < 0.
    \end{equation*}
%%%%%%%%%%%%%%%%%%%%%%%%%%%

% \item If \( i \in \mathcal{P}_2 \), then \( s^T \nabla \obj{i}(\bar{z}) > 0 \) and for \( 0 < \muc < 1 \), \( \rhoc > \max \left\{ 0, \frac{ \muc \cdot s^T \nabla \obj{i}(\bar{z}) \cdot \|\bar{z} - \xlm\| }{ s^T (\bar{z} - \xlm) } \right\} \), it follows that

% \item If \( i \in \mathcal{P}_2 \), then \( s^T \nabla \obj{i}(\bar{z}) > 0 \) and for $0 < \muc < \min \left\{1, \frac{\rhoc}{\|\bar{z} - \xlm\|} \cdot \frac{s^T (\bar{z} - \xlm)}{\max_{i \in \mathcal{P}_2} \left( s^T \nabla \obj{i}(\bar{z}) \right)} \right\}$ and $\rhoc > 0$, it follows that:

\item If \( i \in \mathcal{P}_2 \), then \( s^T \nabla \obj{i}(\bar{z}) > 0 \) and for
\[
0 < \muc < \min \left\{
1,\ 
\frac{\rhoc}{\|\bar{z} - \xlm\|} \cdot 
\frac{s^T (\bar{z} - \xlm)}{
\max_{i \in \mathcal{P}_2} \left( s^T \nabla \obj{i}(\bar{z}) \right)}
\right\} \text{~with~} \rhoc > 0,
\]
it follows that:

\begin{align*}
    s^T \nabla \gdf_{i,\muc,\rhoc,\xlm}(\bar{z}) 
    &= \muc \cdot s^T \nabla \obj{i}(\bar{z}) - \rhoc \frac{s^T (\bar{z} - \xlm)}{\|\bar{z} - \xlm\|} < 0.
\end{align*}

    In particular, from (\ref{eq1.8}), we have \( \|\nabla \obj{i}(\bar{z})\| \leq \mathcal{M} \). Choosing \( s = \bar{z} - \xlm \), the following estimate holds:
\begin{align*}
    0 < \muc &< \min \left\{ 1, \frac{\rhoc}{\mathcal{M}} \right\} \\
    &\leq \min \left\{ 1, \frac{\rhoc}{\|\bar{z} - \xlm\|} \cdot \frac{\|\bar{z} - \xlm\|^2}{\|\bar{z} - \xlm\| \cdot \|\nabla \obj{i}(\bar{z})\|} \right\} \\
    &\leq \min \left\{ 1, \frac{\rhoc}{\|\bar{z} - \xlm\|} \cdot \frac{s^T(\bar{z} - \xlm)}{\max_i \left( s^T \nabla \obj{i}(\bar{z}) \right)} \right\},
\end{align*}
which demonstrates that, under the conditions \( 0 < \muc < \min \left\{ 1, \frac{\rhoc}{\mathcal{M}} \right\} \), \( \rhoc > 0 \) and for \( s = \bar{z} - \xlm \), $s^T \nabla \gdf_{i,\muc,\rhoc,\xlm}(\bar{z})<0$.

\item If \( i \in \mathcal{P}_3 \), then \( s^T \nabla \obj{i}(\bar{z}) < 0 \) and for \( 0 < \muc < 1 \), \( \rhoc > 0 \), it follows that:
\begin{align*}
    s^T \nabla \gdf_{i,\muc,\rhoc,\xlm}(\bar{z}) 
    &= \muc \cdot s^T \nabla \obj{i}(\bar{z}) - \rhoc \frac{s^T (\bar{z} - \xlm)}{\|\bar{z} - \xlm\|} < 0.
\end{align*}
\end{enumerate}
\par Consider {\bm{$j \in \mathcal{I}_2$}}. Then \( \obj{j}(z) > \obj{j}(\xlm) \) holds for all \( j \in \mathcal{I}_2 \), and by Property~(\ref{property4_lemma3.1}) of Lemma~\ref{lemma3.1}, it follows that:
\[
A'_{\muc}(\obj{j}(\bar{z}) - \obj{j}(\xlm)) \leq \muc.
\]
Accordingly, the following cases arise:

%%%%%%%%%%%%%%%%%%%%%%%%%%%
\begin{enumerate}[label=\textbf{Case 2\alph*.}]

\item If \( i \in \mathcal{Q}_1 \), then \( s^T \nabla \obj{i}(\bar{z}) = 0 \) and for \( 0 < \muc < 1 \), \( \rhoc > 0 \), it follows that:
\begin{equation*}
    s^T \nabla \gdf_{i,\muc,\rhoc,\xlm}(\bar{z}) = - \rhoc \frac{s^T (\bar{z} - \xlm)}{\|\bar{z} - \xlm\|} < 0.
\end{equation*}

\item If \( i \in \mathcal{Q}_2 \), then \( s^T \nabla \obj{i}(\bar{z}) > 0 \) and for
\[
0 < \muc < \min \left\{
1,\ 
\frac{\rhoc}{\|\bar{z} - \xlm\|} \cdot 
\frac{s^T (\bar{z} - \xlm)}{
\max_{i \in \mathcal{Q}_2} \left( s^T \nabla \obj{i}(\bar{z}) \right)}
\right\} \text{~with~} \rhoc > 0,
\]
it follows that:
\begin{align*}
    s^T \nabla \gdf_{i,\muc,\rhoc,\xlm}(\bar{z})
     &= A'_{\muc}\left( \obj{i}(\bar{z}) - \obj{i}(\xlm) \right) \cdot s^T \nabla \obj{i}(\bar{z}) - \rhoc \frac{s^T (\bar{z} - \xlm)}{\|\bar{z} - \xlm\|}\\
    &\leq \muc \cdot s^T \nabla \obj{i}(\bar{z}) - \rhoc \frac{s^T (\bar{z} - \xlm)}{\|\bar{z} - \xlm\|} < 0.
\end{align*}
Using (\ref{eq1.8}), we observe that \( \|\nabla \obj{i}(\bar{z})\| \leq \mathcal{M} \). By selecting \( s = \bar{z} - \xlm \), the following bound can be derived:
\begin{align*}
    0 < \muc &< \min\left\{1, \frac{\rhoc}{\mathcal{M}}\right\} \\
    &\leq \min\left\{1, \frac{\rhoc}{\|\bar{z} - \xlm\|} \cdot \frac{\|\bar{z} - \xlm\|^2}{\|\bar{z} - \xlm\| \cdot \|\nabla \obj{i}(\bar{z})\|}\right\} \\
    &\leq \min\left\{1, \frac{\rhoc}{\|\bar{z} - \xlm\|} \cdot \frac{s^T(\bar{z} - \xlm)}{\max_{i \in \mathcal{Q}_2} \left( s^T \nabla \obj{i}(\bar{z}) \right)}\right\}.
\end{align*}
Hence for \( 0 < \muc < \min \left\{ 1, \frac{\rhoc}{\mathcal{M}} \right\} \), \( \rhoc > 0 \), and with \( s = \bar{z} - \xlm \), \( s^T \nabla \gdf_{i,\muc,\rhoc,\xlm}(\bar{z}) < 0 \) is satisfied.

% \item If \( i \in \mathcal{Q}_3 \), then \( s^T \nabla \obj{i}(\bar{z}) < 0 \) and for \( 0 < \muc < 1 \), \( \rhoc > 0 \), it follows that:
% \begin{align*}
%     s^T \nabla \gdf_{i,\muc,\rhoc,\xlm}(\bar{z}) 
%     &= \muc \cdot s^T \nabla \obj{i}(\bar{z}) - \rhoc \frac{s^T (\bar{z} - \xlm)}{\|\bar{z} - \xlm\|} < 0.
% \end{align*}

\item If \( i \in \mathcal{Q}_3 \), then \( s^T \nabla \obj{i}(\bar{z}) < 0 \), then two cases arise:
\begin{enumerate}[label=(\roman*)]
    \item For \( c' = 0 \): If \( 0 < \muc < 1 \) and \( \rhoc > 0 \), then from Lemma~\ref{lemma3.1}, we have:
\begin{align*}
A'_{\muc}(\obj{i}(\bar{z}) - \obj{i}(\xlm)) 
&= V_{\muc}(\obj{i}(\bar{z}) - \obj{i}(\xlm)) \\
&\quad + \big[ \obj{i}(\bar{z}) - \obj{i}(\xlm) \big]
 \cdot V'_{\muc}(\obj{i}(\bar{z}) - \obj{i}(\xlm)) \\
&= \muc > 0,
\end{align*}
    and hence,
        \begin{align*}
        s^T \nabla \gdf_{i,\muc,\rhoc,\xlm}(\bar{z}) 
        &= A'_{\muc}\left( \obj{i}(\bar{z}) - \obj{i}(\xlm) \right) \cdot s^T \nabla \obj{i}(\bar{z}) - \rhoc \frac{s^T (\bar{z} - \xlm)}{\|\bar{z} - \xlm\|} < 0.
    \end{align*}

    \item For \( c' > 0 \): If \( \rhoc > 0 \) and
    \[
    0 < \muc < \min \left\{ 1, \frac{\rhoc}{c' \Lconst \|\bar{z} - \xlm\|^2} \cdot \frac{s^T (\bar{z} - \xlm)}{\max_{i \in \mathcal{Q}_3} \left( - s^T \nabla \obj{i}(\bar{z}) \right)} \right\},
    \]
    then from the definition of \( V_{\muc} \), we have \( V_{\muc}(y) \geq c\muc > 0 \) for all \( y \), and
    \[
    -c'\muc \leq V'_{\muc}(y) \leq 0, \quad \text{for all } y \geq 0.
    \]
    Since \( \obj{i}(\bar{z}) - \obj{i}(\xlm) > 0 \), and by the Lipschitz condition (see Inequality~(\ref{eq1.4})), we have:
    \[
    0 < \obj{i}(\bar{z}) - \obj{i}(\xlm) \leq \Lconst \|\bar{z} - \xlm\|,
    \]
    and therefore,
    \begin{align*}
    A'_{\muc}(\obj{i}(\bar{z}) - \obj{i}(\xlm)) 
    &= V_{\muc}(\obj{i}(\bar{z}) - \obj{i}(\xlm)) \\
   &\quad + \left[\obj{i}(\bar{z}) - \obj{i}(\xlm)\right] \cdot  V'_{\muc}(\obj{i}(\bar{z}) - \obj{i}(\xlm)) \\
    &> 0 + \Lconst \|\bar{z} - \xlm\| \cdot (-c'\muc) = -c'\muc \Lconst \|\bar{z} - \xlm\|.
    \end{align*}
    Consequently,
    \begin{align*}
    s^T \nabla \gdf_{i, \muc, \rhoc, \xlm}(\bar{z}) 
    &= A'_{\muc}(\obj{i}(\bar{z}) - \obj{i}(\xlm)) \cdot s^T \nabla \obj{i}(\bar{z}) - \rhoc \frac{s^T(\bar{z} - \xlm)}{\|\bar{z} - \xlm\|} \\
    &< -c'\muc \Lconst \|\bar{z} - \xlm\| \cdot s^T \nabla \obj{i}(\bar{z}) - \rhoc \frac{s^T(\bar{z} - \xlm)}{\|\bar{z} - \xlm\|} < 0.
    \end{align*}
    Now, from Inequality~(\ref{eq1.7}), we have \( \|\bar{z} - \xlm\| \leq \Kconst \) and from Inequality~(\ref{eq1.8}), \( \|\nabla \obj{i}(\bar{z})\| \leq \mathcal{M} \). If \( s = \bar{z} - \xlm \), then:
    \begin{align*}
    0 < \muc &< \min \left\{ 1, \frac{\rhoc}{c' \Lconst \Kconst \mathcal{M}} \right\} \\
    &\leq \min \left\{ 1, \frac{\rhoc}{c' \Lconst \|\bar{z} - \xlm\|^2} \cdot \frac{\|\bar{z} - \xlm\|^2}{\|\bar{z} - \xlm\| \cdot \|\nabla \obj{i}(\bar{z})\|} \right\} \\
    &\leq \min \left\{ 1, \frac{\rhoc}{c' \Lconst \|\bar{z} - \xlm\|^2} \cdot \frac{s^T(\bar{z} - \xlm)}{\max_i \left( - s^T \nabla \obj{i}(\bar{z}) \right)} \right\}.
    \end{align*}

    It is important to note that if \( c' \) or \( \lim_{\muc \to 0} c' \) is a positive constant, then for sufficiently small \( \muc \), one can always find \( \muc \in \left(0, \min \left\{1, \frac{\rhoc}{c' \Lconst \Kconst \mathcal{M}} \right\} \right) \) satisfying the above condition. Furthermore, if \( \lim_{\muc \to 0} c' = \infty \), then by the definition of \( c' \), we have \( \lim_{\muc \to 0} c'\muc = 0 \), which ensures the existence of a sufficiently small \( \muc \) such that \( 0 < c'\muc < \frac{\rhoc}{\Lconst \Kconst \mathcal{M}} \).
    \end{enumerate}
    
\end{enumerate}

\par Consider {\bm{$j \in \mathcal{I}_3$}}. Since \( \obj{j}(\bar{z}) < \obj{j}(\xlm) \) holds, it follows from the definitions of \( A_{\muc} \) and \( V_{\muc} \) that:
\[
A'_{\muc} \left( \obj{i}(\bar{z}) - \obj{i}(\xlm) \right) \geq c\muc > 0.
\]
Accordingly, the following cases arise:

\begin{enumerate}[label=\textbf{Case 3\alph*.}]
    \item If \( i \in \mathcal{R}_1 \), then \( s^T \nabla \obj{i}(\bar{z}) = 0 \) and for \( 0 < \muc < 1 \), \( \rhoc > 0 \), it follows that:
    \begin{equation*}
        s^T \nabla \gdf_{i,\muc,\rhoc,\xlm}(\bar{z}) = - \rhoc \frac{s^T (\bar{z} - \xlm)}{\|\bar{z} - \xlm\|} < 0.
    \end{equation*}

    \item If \( i \in \mathcal{R}_2 \), then \( s^T \nabla \obj{i}(\bar{z}) > 0 \) and for
    \begin{equation*}
    0 < \muc < 1 \quad \text{and} \quad \rhoc > \max \left\{ 0,\frac{\mathcal{C} \cdot \max_{i \in \mathcal{R}_2} \left(s^T \nabla \obj{i}(\bar{z})\right) \cdot \|\bar{z} - \xlm\|}{s^T (\bar{z} - \xlm)} \right\},
    \end{equation*}
    it follows that:
    \begin{align*}
        s^T \nabla \gdf_{i,\muc,\rhoc,\xlm}(\bar{z}) 
        &= A'_{\muc}\left( \obj{i}(\bar{z}) - \obj{i}(\xlm) \right) \cdot s^T \nabla \obj{i}(\bar{z}) - \rhoc \frac{s^T (\bar{z} - \xlm)}{\|\bar{z} - \xlm\|}\\
         &= \mathcal{C} \cdot s^T \nabla \obj{i}(\bar{z}) - \rhoc \frac{s^T (\bar{z} - \xlm)}{\|\bar{z} - \xlm\|} < 0.
    \end{align*}

\item If \( i \in \mathcal{R}_3 \), then \( s^T \nabla \obj{i}(\bar{z}) < 0 \) and for \( 0 < \muc < 1 \), \( \rhoc > 0 \), it follows that:
\begin{align*}
    s^T \nabla \gdf_{i,\muc,\rhoc,\xlm}(\bar{z}) 
    &= A'_{\muc}\left( \obj{i}(\bar{z}) - \obj{i}(\xlm) \right) \cdot s^T \nabla \obj{i}(\bar{z}) - \rhoc \frac{s^T (\bar{z} - \xlm)}{\|\bar{z} - \xlm\|} < 0.
\end{align*}
\end{enumerate}
\par Consequently, when {\bm{$j \in \mathcal{I}_3$}} and for
\begin{equation*}
0 < \muc < 1 \quad \text{and} \quad \rhoc > \max \left\{ 0,\frac{\mathcal{C} \cdot \max_{j \in \mathcal{I}_3} \left(s^T \nabla \obj{j}(\bar{z})\right) \cdot \|\bar{z} - \xlm\|}{s^T (\bar{z} - \xlm)} \right\},    
\end{equation*}
it follows that:
\begin{equation*}
    s^T \nabla \gdf_{j,\muc,\rhoc,\xlm}(\bar{z}) < 0.
\end{equation*}
\par By exhaustively analyzing all possible cases corresponding to the partitions of \( \Idx{m} \), it is established that \( s^T \nabla \gdf_{i,\muc,\rhoc,\xlm}(\bar{z}) \) remains strictly negative under appropriate selections of $\muc$ and $\rhoc$. Consequently, the direction \( s \) qualifies as a valid descent direction at \( \bar{z} \), thereby completing the proof.
\end{proof}

%%%%%%%%%%%%%%%%%%%%%%%%%%%%%%%%%%%%%%%%%%%%%%%%%
\begin{theorem}\label{Lemma3.4}
Let \( \mathcal{I}_1 \), \( \mathcal{I}_2 \), and \( \mathcal{I}_3 \) be three partitions of \( \Idx{m} \) such that
\begin{enumerate}
    \item For all \( j \in \mathcal{I}_1 \), \( \obj{j}(z) = \obj{j}(\xlm) \);
    \item For all \( j \in \mathcal{I}_2 \), \( \obj{j}(z) > \obj{j}(\xlm) \);
    \item For all \( j \in \mathcal{I}_3 \), \( \obj{j}(z) < \obj{j}(\xlm) \).
\end{enumerate}

% Consider \( \bar{z} \) satisfying 
% \[
% \bar{z} \in \left\{ x \in \text{int}(\Feas)\setminus \xlm : x \text{ is not critical for } \fvec,~ \obj{j}(z) \geq \obj{j}(\xlm) \text{ for at least one } j \right\}.
% \]

Consider $\bar{z} \in int(\Feas) \setminus \xlm$ such that  $\bar{z}$ is not critical for $\fvec$, and $\obj{j}(\bar{z}) \geq \obj{j}(\xlm)$ holds for at least one $j$. Let \( s \) be a feasible direction satisfying \( s^T (\bar{z} - \xlm) > 0 \) and for the partitions \( \mathcal{Q}_1, \mathcal{Q}_2, \mathcal{Q}_3 \) of \( \mathcal{I}_2 \), and partitions \( \mathcal{R}_1, \mathcal{R}_2, \mathcal{R}_3 \) of \( \mathcal{I}_3 \), the following conditions hold:

\begin{enumerate}[label=2\alph*.]
    \item \( s^T \nabla \obj{i}(\bar{z}) = 0 \), for all \( i \in \mathcal{Q}_1 \);\label{case_2a_2nd lemma}
    \item \( s^T \nabla \obj{i}(\bar{z}) > 0 \), for all \( i \in \mathcal{Q}_2 \);\label{case_2b_2nd lemma}
    \item \( s^T \nabla \obj{i}(\bar{z}) < 0 \), for all \( i \in \mathcal{Q}_3 \);\label{case_2c_2nd lemma}
\end{enumerate}
\begin{enumerate}[label=3\alph*.]
    \item \( s^T \nabla \obj{i}(\bar{z}) = 0 \), for all \( i \in \mathcal{R}_1 \);\label{case_3a_2nd lemma}
    \item \( s^T \nabla \obj{i}(\bar{z}) > 0 \), for all \( i \in \mathcal{R}_2 \);\label{case_3b_2nd lemma}
    \item \( s^T \nabla \obj{i}(\bar{z}) < 0 \), for all \( i \in \mathcal{R}_3 \).\label{case_3c_2nd lemma}
\end{enumerate}
\begin{equation*}
\begin{aligned}
\text{For } \rhoc > \max \Bigg\{ 0,\ 
& \frac{
    \mathcal{C} \cdot 
    \max_{i \in \mathcal{R}_2} \left(s^T \nabla \obj{i}(\bar{z})\right) \cdot 
    \|\bar{z} - \xlm\|
}{
    s^T (\bar{z} - \xlm)
}
\Bigg\} \quad \text{and} \\
0 < \muc < \min \Bigg\{ 1,\ \frac{\rhoc}{\mathcal{M}},\ 
& \frac{\rhoc}{\|\bar{z} - \xlm\|} \cdot 
  \frac{s^T (\bar{z} - \xlm)}{
    \max_{i \in \mathcal{Q}_2} \left(s^T \nabla \obj{i}(\bar{z})\right)}
\Bigg\} \quad \text{for } c' = 0, \\
\text{or~~} 0 < \muc < \min \Bigg\{ 1,\ \frac{\rhoc}{\mathcal{M}},\ 
& \frac{\rhoc}{\|\bar{z} - \xlm\|} \cdot 
  \frac{s^T (\bar{z} - \xlm)}{
    \max_{i \in \mathcal{Q}_2} \left(s^T \nabla \obj{i}(\bar{z})\right)}, \\
& \frac{\rhoc}{c' \Lconst \|\bar{z} - \xlm\|^2} \cdot 
  \frac{s^T (\bar{z} - \xlm)}{
    \max_{i \in \mathcal{Q}_3} \left(- s^T \nabla \obj{i}(\bar{z})\right)}
\Bigg\} \quad \text{for } c' > 0,
\end{aligned}
\end{equation*}

the direction \( s \) constitutes a descent direction of \( \gdf_{\muc, \rhoc, \xlm} \) at \( \bar{z} \).

%%%%%%%%%%%%%%%%%%%%%%%%%%%%%%%%%%%

\end{theorem}

\begin{proof}
As previously noted in Theorem \ref{Lemma3.3}, the proof relies on partitioning the index set \( \Idx{m} \) into three cases. For each, a feasible direction \( d \) with \( s^T(\bar{z} - \xlm) > 0 \) is shown to be a descent direction under suitable choices of \( \rhoc \) and \( \muc \).
\par Consider {\bm{$j \in \mathcal{I}_1$}}. Since \( \bar{z} \) is not a critical point of the original objective vector \( \fvec{} \), we have
\[
\sum_{j=1}^m \mult_j \nabla \obj{j}(\bar{z}) \neq 0, \quad \mult \in \Rplus{m}, \ \mult \neq \mathbf{0}^{m}.
\]

Suppose, for the sake of contradiction, that \( \bar{z} \) is a critical point of the smoothed scalarization \( \gdf_{\muc, \rhoc, \xlm} \). Then, there exists a vector \( \bar{\eta} \in \Rplus{m} \), with \( \bar{\eta} \neq \mathbf{0}^{m} \), such that
\[
\sum_{j=1}^m \bar{\eta}_j \nabla \gdf_{j,\muc,\rhoc,\xlm}(\bar{z}) = 0.
\]

Invoking Lemma~\ref{lemma3.1}, and noting that \( \obj{j}(\bar{z}) = \obj{j}(\xlm) \) for all \( j \in \mathcal{I}_1 \), we obtain
\[
A'_{\muc}(\obj{j}(\bar{z}) - \obj{j}(\xlm)) = A'_{\muc}(0) = \muc.
\]
Substituting into the gradient expression yields
\begin{align*}
    0 &= \sum_{j=1}^m \bar{\eta}_j \nabla \gdf_{j,\muc,\rhoc,\xlm}(\bar{z}) \\
      &= \sum_{j=1}^m A'_{\muc}(0) \cdot \bar{\eta}_j \nabla \obj{j}(\bar{z}) - \rhoc \sum_{j=1}^m \bar{\eta}_j \frac{\bar{z} - \xlm}{\|\bar{z} - \xlm\|} \\
      &= \muc \sum_{j=1}^m \bar{\eta}_j \nabla \obj{j}(\bar{z}) - \rhoc \sum_{j=1}^m \bar{\eta}_j \frac{\bar{z} - \xlm}{\|\bar{z} - \xlm\|}.
\end{align*}

Rearranging, we obtain the identity
\[
\muc = \frac{\rhoc q}{\left\| \sum_{j=1}^m \bar{\eta}_j \nabla \obj{j}(\bar{z}) \right\|}, \quad \text{where } q = \sum_{j=1}^m \bar{\eta}_j.
\]

Now, from Inequality (\ref{eq1.8}), it follows that:
\[
\left\| \sum_{j=1}^m \bar{\eta}_j \nabla \obj{j}(\bar{z}) \right\| \leq \mathcal{M}q,
\]
which in turn implies:
\[
\muc \geq \frac{\rhoc}{\mathcal{M}}.
\]

Therefore, if \( \rhoc > 0 \) and \( 0 < \muc < \min \left\{1, \frac{\rhoc}{\mathcal{M}} \right\} \), the assumed criticality condition leads to a contradiction. Thus,
\[
\sum_{j=1}^m \bar{\eta}_j \nabla \gdf_{j,\muc,\rhoc,\xlm}(\bar{z}) \neq 0, \quad \text{for all } j \in \mathcal{I}_1,
\]
which confirms that \( \bar{z} \) is not a critical point of \( \gdf_{\muc, \rhoc, \xlm} \) under the stated conditions.

Consider {\bm{$j \in \mathcal{I}_2$}}. Then \( \obj{j}(z) > \obj{j}(\xlm) \) holds for all \( j \in \mathcal{I}_2 \), and by Property~(\ref{property4_lemma3.1}) of Lemma~\ref{lemma3.1}, it follows that:
\[
A'_{\muc}(\obj{j}(\bar{z}) - \obj{j}(\xlm)) \leq \muc.
\]
Accordingly, the following cases arise:

\begin{enumerate}[label=\textbf{Case 2\alph*.}]

\item If \( i \in \mathcal{Q}_1 \), then \( s^T \nabla \obj{i}(\bar{z}) = 0 \) and  For \( 0 < \muc < 1 \) and \( \rhoc > 0 \), it follows that:
    \begin{equation*}
        s^T \nabla \gdf_{i,\muc,\rhoc,\xlm}(\bar{z}) = - \rhoc \frac{s^T (\bar{z} - \xlm)}{\|\bar{z} - \xlm\|} < 0.
    \end{equation*}
%%%%%%%%%%%%%%%%%%%%%%%%%%%%%%%%%%%%%

    \item If \( i \in \mathcal{Q}_2 \), where \( \obj{i}(\bar{z}) > \obj{i}(\xlm) \) and \( s^T \nabla \obj{i}(\bar{z}) > 0 \), then for $0 < \muc < \min \left\{1, \frac{\rhoc}{\|\bar{z} - \xlm\|} \cdot \frac{s^T (\bar{z} - \xlm)}{\max_{i \in \mathcal{Q}_2} \left( s^T \nabla \obj{i}(\bar{z}) \right)} \right\}$ and $\rhoc > 0$, it follows that:
    \begin{align*}
        s^T \nabla \gdf_{i,\muc,\rhoc,\xlm}(\bar{z}) 
        &= A'_{\muc}(\obj{i}(\bar{z}) - \obj{i}(\xlm)) \cdot d^{T} \nabla \obj{i}(\bar{z}) - \rhoc \frac{s^T (\bar{z} - \xlm)}{\|\bar{z} - \xlm\|} \\
        &\leq \muc \cdot s^T \nabla \obj{i}(\bar{z}) - \rhoc \frac{s^T (\bar{z} - \xlm)}{\|\bar{z} - \xlm\|} < 0.
    \end{align*}

    In particular, from (\ref{eq1.8}) we have \( \|\nabla \obj{i}(\bar{z})\| \leq \mathcal{M} \). Choosing \( s = \bar{z} - \xlm \), the following estimate holds:
\begin{align*}
    0 < \muc &< \min \left\{ 1, \frac{\rhoc}{\mathcal{M}} \right\} \\
    &\leq \min \left\{ 1, \frac{\rhoc}{\|\bar{z} - \xlm\|} \cdot \frac{\|\bar{z} - \xlm\|^2}{\|\bar{z} - \xlm\| \cdot \|\nabla \obj{i}(\bar{z})\|} \right\} \\
    &\leq \min \left\{ 1, \frac{\rhoc}{\|\bar{z} - \xlm\|} \cdot \frac{s^T(\bar{z} - \xlm)}{\max_{i \in \mathcal{Q}_2} \left( s^T \nabla \obj{i}(\bar{z}) \right)} \right\},
\end{align*}
which demonstrates that, under the conditions \( 0 < \muc < \min \left\{ 1, \frac{\rhoc}{\mathcal{M}} \right\} \), \( \rhoc > 0 \) and for \( s = \bar{z} - \xlm \), $s^T \nabla \gdf_{i,\muc,\rhoc,\xlm}(\bar{z})<0$.

%%%%%%%%%%%%%%%%%%%%%

\item If \( i \in \mathcal{Q}_3 \), where \( \obj{i}(\bar{z}) > \obj{i}(\xlm) \) and \( s^T \nabla \obj{i}(\bar{z}) < 0 \), then two cases arise:
\begin{enumerate}[label=(\roman*)]
    \item For \( c' = 0 \): If \( \rhoc > 0 \) and \( 0 < \muc < 1 \), then from Lemma~\ref{lemma3.1}, we have:
        \begin{align*}
    A'_{\muc}(\obj{i}(\bar{z}) - \obj{i}(\xlm)) 
    &= V_{\muc}(\obj{i}(\bar{z}) - \obj{i}(\xlm)) \\
   &\quad + \left[\obj{i}(\bar{z}) - \obj{i}(\xlm)\right] \cdot  V'_{\muc}(\obj{i}(\bar{z}) - \obj{i}(\xlm)) \\
     &= \muc > 0,
    \end{align*}
    and hence,
        \begin{align*}
        s^T \nabla \gdf_{i,\muc,\rhoc,\xlm}(\bar{z}) 
        &= A'_{\muc}\left( \obj{i}(\bar{z}) - \obj{i}(\xlm) \right) \cdot s^T \nabla \obj{i}(\bar{z}) - \rhoc \frac{s^T (\bar{z} - \xlm)}{\|\bar{z} - \xlm\|} < 0.
    \end{align*}

    \item For \( c' > 0 \): If \( \rhoc > 0 \) and
    \[
    0 < \muc < \min \left\{ 1, \frac{\rhoc}{c' \Lconst \|\bar{z} - \xlm\|^2} \cdot \frac{s^T (\bar{z} - \xlm)}{\max_{i \in \mathcal{Q}_3} \left( - s^T \nabla \obj{i}(\bar{z}) \right)} \right\},
    \]
    then from the definition of \( V_{\muc} \), we have \( V_{\muc}(y) \geq c\muc > 0 \) for all \( y \), and
    \[
    -c'\muc \leq V'_{\muc}(y) \leq 0, \quad \text{for all } y \geq 0.
    \]
    Since \( \obj{i}(\bar{z}) - \obj{i}(\xlm) > 0 \), and by the Lipschitz condition (see Inequality~(\ref{eq1.4})), we have:
    \[
    0 < \obj{i}(\bar{z}) - \obj{i}(\xlm) \leq \Lconst \|\bar{z} - \xlm\|,
    \]
    and therefore,
    \begin{align*}
    A'_{\muc}(\obj{i}(\bar{z}) - \obj{i}(\xlm)) 
    &= V_{\muc}(\obj{i}(\bar{z}) - \obj{i}(\xlm)) \\
   &\quad + \left[\obj{i}(\bar{z}) - \obj{i}(\xlm)\right] \cdot  V'_{\muc}(\obj{i}(\bar{z}) - \obj{i}(\xlm)) \\
    &> 0 + \Lconst \|\bar{z} - \xlm\| \cdot (-c'\muc) = -c'\muc \Lconst \|\bar{z} - \xlm\|.
    \end{align*}
    Consequently,
    \begin{align*}
    s^T \nabla \gdf_{i, \muc, \rhoc, \xlm}(\bar{z}) 
    &= A'_{\muc}(\obj{i}(\bar{z}) - \obj{i}(\xlm)) \cdot s^T \nabla \obj{i}(\bar{z}) - \rhoc \frac{s^T(\bar{z} - \xlm)}{\|\bar{z} - \xlm\|} \\
    &< -c'\muc \Lconst \|\bar{z} - \xlm\| \cdot s^T \nabla \obj{i}(\bar{z}) - \rhoc \frac{s^T(\bar{z} - \xlm)}{\|\bar{z} - \xlm\|} < 0.
    \end{align*}

    Now, from (\ref{eq1.7}), we have \( \|\bar{z} - \xlm\| \leq \Kconst \) and from (\ref{eq1.8}), \( \|\nabla \obj{i}(\bar{z})\| \leq \mathcal{M} \). If \( s = \bar{z} - \xlm \), then:
    \begin{align*}
    0 < \muc &< \min \left\{ 1, \frac{\rhoc}{c' \Lconst \Kconst \mathcal{M}} \right\} \\
    &\leq \min \left\{ 1, \frac{\rhoc}{c' \Lconst \|\bar{z} - \xlm\|^2} \cdot \frac{\|\bar{z} - \xlm\|^2}{\|\bar{z} - \xlm\| \cdot \|\nabla \obj{i}(\bar{z})\|} \right\} \\
    &\leq \min \left\{ 1, \frac{\rhoc}{c' \Lconst \|\bar{z} - \xlm\|^2} \cdot \frac{s^T(\bar{z} - \xlm)}{\max_{i \in \mathcal{Q}_3} \left( - s^T \nabla \obj{i}(\bar{z}) \right)} \right\}.
    \end{align*}
Since a similar analysis has been employed earlier, we briefly note that if \( c' \) or \( \lim_{\muc \to 0} c' \) is a positive constant, then one can select \linebreak
\( \muc \in \left(0, \min \left\{1, \frac{\rhoc}{c' \Lconst \Kconst \mathcal{M}} \right\} \right) \) for sufficiently small \( \muc \). Alternatively, if \( c' \to \infty \) as \( \muc \to 0 \), then by definition of \( c' \), \( \lim_{\muc \to 0} c'\muc = 0 \), ensuring the existence of a sufficiently small \( \muc \) such that \( 0 < c'\muc < \frac{\rhoc}{\Lconst \Kconst \mathcal{M}} \).

\end{enumerate}

\end{enumerate}
% For \( j \in \mathcal{I}_3 \), the following cases arise.\\

Consider {\bm{$j \in \mathcal{I}_3$}}. Since \( \obj{j}(\bar{z}) < \obj{j}(\xlm) \) holds, it follows from the definitions of \( A_{\muc} \) and \( V_{\muc} \) that:
\[
A'_{\muc} \left( \obj{i}(\bar{z}) - \obj{i}(\xlm) \right) \geq c\muc > 0.
\]
Accordingly, the following cases arise:

\begin{enumerate}[label=\textbf{Case 3\alph*.}]

\item If \( i \in \mathcal{R}_1 \), then \( s^T \nabla \obj{i}(\bar{z}) = 0 \), and for \( \rhoc > 0 \), \( 0 < \muc < 1 \), it follows that:
\begin{align*}
    s^T \nabla \gdf_{i, \muc, \rhoc, \xlm} (\bar{z}) 
    &= A'_{\muc}(\obj{i}(\bar{z}) - \obj{i}(\xlm)) \cdot s^T \nabla \obj{i}(\bar{z}) - \rhoc \frac{s^T (\bar{z} - \xlm)}{\|\bar{z} - \xlm\|} \\
    &= - \rhoc \frac{s^T (\bar{z} - \xlm)}{\|\bar{z} - \xlm\|} < 0.
\end{align*}

\item If \( i \in \mathcal{R}_2 \), then \( s^T \nabla \obj{i}(\bar{z}) > 0 \) and for
\[
0 < \muc < 1, \quad \rhoc > \max \left\{ 0, \frac{\mathcal{C} \cdot \max_{i \in \mathcal{R}_2} \left( s^T \nabla \obj{i}(\bar{z}) \right) \cdot \|\bar{z} - \xlm\|}{s^T (\bar{z} - \xlm)} \right\},
\]
it follows that:
    \begin{align*}
        s^T \nabla \gdf_{i,\muc,\rhoc,\xlm}(\bar{z}) 
        &= A'_{\muc}\left( \obj{i}(\bar{z}) - \obj{i}(\xlm) \right) \cdot s^T \nabla \obj{i}(\bar{z}) - \rhoc \frac{s^T (\bar{z} - \xlm)}{\|\bar{z} - \xlm\|}\\
         &= \mathcal{C} \cdot s^T \nabla \obj{i}(\bar{z}) - \rhoc \frac{s^T (\bar{z} - \xlm)}{\|\bar{z} - \xlm\|} < 0,
    \end{align*}
where \( C = A'_{\muc} \left( \obj{i}(\bar{z}) - \obj{i}(\xlm) \right) \).
%%%%%%%%
\item

If \( i \in \mathcal{R}_3 \), then \( s^T \nabla \obj{i}(\bar{z}) < 0 \) and for any \( \rhoc > 0 \) and \( 0 < \muc < 1 \), it follows:
\begin{align*}
    s^T \nabla \gdf_{i, \mu, \rhoc, \xlm} (\bar{z}) 
    &= A'_{\muc} \left(\obj{i}(\bar{z}) - \obj{i}(\xlm)\right) \cdot s^T \nabla \obj{i}(\bar{z}) 
    - \rhoc \frac{s^T (\bar{z} - \xlm)}{\|\bar{z} - \xlm\|} < 0.
\end{align*}

\end{enumerate}

\par Therefore, when {\bm{$j \in \mathcal{I}_3$}} and for
\begin{equation*}
0 < \muc < 1 \quad \text{and} \quad \rhoc > \max \left\{ 0,\frac{\mathcal{C} \cdot \max_{j \in \mathcal{I}_3} \left(s^T \nabla \obj{j}(\bar{z})\right) \cdot \|\bar{z} - \xlm\|}{s^T (\bar{z} - \xlm)} \right\},    
\end{equation*}
it follows that:
\begin{equation*}
    s^T \nabla \gdf_{j,\muc,\rhoc,\xlm}(\bar{z}) < 0.
\end{equation*}
\par Accordingly, the theorem guarantees that, with suitable choices of \( \muc \) and \( \rhoc \), the vector \( s \) serves as a valid descent direction for the function \( \gdf_{\muc, \rhoc, \xlm} \) evaluated at \( \bar{z} \).

\end{proof}

%%%%%%%%%%%%%%%%%%%%%%%%%%%%%%%%%%%%%%%%%%%%%%%%%%%%%%%%%%%%%%%%%%%%

\begin{theorem}\label{theorem_3.5}
If
\[
\rhoc > \max \left\{ 0,\frac{\mathcal{C} \cdot \max_{j \in \mathcal{I}_3} \left(s^T \nabla \obj{j}(z)\right) \cdot \|z - \xlm\|}{s^T (x - \xlm)} \right\},
\]
where \( \mathcal{C} = A'_{\muc} \left( \obj{i}(z) - \obj{i}(\xlm) \right) \), then the function \( \gdf_{\muc, \rhoc, \xlm} \) admits no stationary point in the set \( \hat{\Feas}(\xlm) \), provided that \( \muc \) satisfies the following conditions, depending on the value of \( c' \):
\begin{itemize}
    \item if \( c' = 0 \), then
    \[
    0 < \muc < \min \left\{
        1,\,
        \frac{\rhoc}{\mathcal{M}}
    \right\};
    \]

    \item if \( c' > 0 \), then
    \[
    0 < \muc < \min \left\{
        1,\,
        \frac{\rhoc}{\mathcal{M}},\, \frac{\rhoc}{c' \Lconst \Kconst \mathcal{M}}
    \right\}.
    \]
\end{itemize}
\end{theorem}

\begin{proof}
Let \( z \in \hat{\Feas}(\xlm) \) and define \( s = z - \xlm \). By Theorems~\ref{Lemma3.3} and~\ref{Lemma3.4}, it follows that under the stated bounds on \( \rhoc \) and \( \muc \), the direction \( s \) is a valid descent direction for \( \gdf_{\muc, \rhoc, \xlm} \) at \( z \). Therefore, \( z \) cannot be a stationary point. Since this holds for any \( z \in \hat{\Feas}(\xlm) \), it follows that \( \gdf_{\muc, \rhoc, \xlm} \) admits no stationary point in \( \hat{\Feas}(\xlm) \).
\end{proof}

%%%%%%%%%%%%%%%%%%%%%%%%%%%%%%%%%%%%%%%%
%\subsection{Satisfaction of condition (C3)}
\begin{theorem}\label{theorem_3.6}
Let \( z^{\dagger\dagger} \in \operatorname{int}(\Feas) \) be a local weak efficient solution of \( \fvec \) with \( \obj{j}(z^{\dagger\dagger}) < \obj{j}(\xlm) \), $j \in \Idx{m}$. If \( \rhoc > 0 \) is sufficiently small and \( 0 < \muc < 1 \), then the function \( \gdf_{\muc, \rhoc, \xlm} \) has a weak efficient solution \( z' \in \Feas \) such that:
\[
z' \in N(z^{\dagger\dagger},\epsilon) \subset \Feas \quad \text{and} \quad \obj{j}(z') < \obj{j}(\xlm) - \tau,
\]
where \( N(z^{\dagger\dagger},\epsilon) \) is an \( \epsilon \)-neighborhood of \( z^{\dagger\dagger} \) with \( \obj{j}(z) < \obj{j}(\xlm) \) for all \( z \in N(z^{\dagger\dagger},\epsilon) \).
\end{theorem}
\begin{proof}
% 1. \textbf{Compactness and continuity of \( \fvec \):}  
Since \( \Feas \) is compact (Assumption \ref{assumption1}), connected,  \( \obj{j} \) are continuous over \( \Feas \) (Assumption \ref{assumption2}), and \( \obj{j}(z^{\dagger\dagger}) < \obj{j}(\xlm) \), for sufficiently small \( \epsilon > 0 \), there exists a neighborhood \( N(z^{\dagger\dagger},\epsilon) \subset \Feas \) such that:
\[
\obj{j}(z^{\dagger\dagger}) \leq \obj{j}(z) < \obj{j}(\xlm), \quad \forall z \in N(z^{\dagger\dagger},\epsilon).
\]
From (\ref{eq2.2}), $\obj{j}(z^{\dagger\dagger})  < \obj{j}(\xlm)$ implies:
\begin{equation*}
    \obj{j}(z^{\dagger\dagger}) < \obj{j}(\xlm) - \tau.
\end{equation*}
If \( \delta > 0 \) is sufficiently small, then:
\begin{equation}\label{eq10}
    \obj{j}(z^{\dagger\dagger}) + \delta \leq \obj{j}(\xlm) - \tau.
\end{equation}
% 2. \textbf{Lower and boundary sets:}  
Define:
\begin{align}
L &= \{ z \in N(z^{\dagger\dagger},\epsilon) : \obj{j}(z) \leq \obj{j} (z^{\dagger\dagger}) + \delta \}, \label{eq11}\\
B &= \{ z \in N(z^{\dagger\dagger},\epsilon) : \obj{j}(z) = \obj{j}(z^{\dagger\dagger}) + \delta \}.\label{eq12}
\end{align}
Both \( L \) and \( B \) are compact subsets of \( \Feas \).
\par We now show that for $\rhoc > 0$ sufficiently small and $0< \muc<1$, the following inequality holds:
\begin{equation} \label{eq13}
    \gdf_{j,\muc, \rhoc, \xlm}(z^{\dagger\dagger}) <\gdf_{j,\muc, \rhoc, \xlm}(\bar{z}), \quad \forall \bar{z} \in B.
\end{equation}
For any \( \bar{z} \in B \) and using (\ref{eq10}) and (\ref{eq12}), we have:
\begin{equation*}
    \obj{j}(z^{\dagger\dagger}) - \obj{j}(\xlm) <  \obj{j}(z^{\dagger\dagger}) + \delta - \obj{j}(\xlm) = \obj{j}(\bar{z})-\obj{j}(\xlm) \leq -\tau.
\end{equation*}
So, using the Property \ref{property3_lemma3.1} of Lemma \ref{lemma3.1}, we get:
\begin{equation}\label{eq14}
 0 < \delta < A_{\muc}(\obj{j}(\bar{z}) - \obj{j}(\xlm)) - A_{\muc}(\obj{j}(z^{\dagger\dagger}) - \obj{j}(\xlm)).
\end{equation}
Now consider the following two cases:
\begin{enumerate}[label=\textbf{Case \arabic*.}]
\item \( \|\bar{z} - \xlm\| \leq \|z^{\dagger\dagger} - \xlm\| \).\\
Then for $\rhoc > 0$, $0< \muc<1$ and using Inequality (\ref{eq14}), we have:
    \begin{align*}
        \gdf_{j,\muc, \rhoc, \xlm}(z^{\dagger\dagger}) 
        &= A_{\muc}(\obj{j}(z^{\dagger\dagger}) - \obj{j}(\xlm)) - \rhoc \|z^{\dagger\dagger} - \xlm\|  \\
        &< A_{\muc}(\obj{j}(\bar{z}) - \obj{j}(\xlm)) - \rhoc \|\bar{z} - \xlm\| = \gdf_{j,\muc, \rhoc, \xlm}(\bar{z}).
    \end{align*}
%%%%%%%%%%%%%%%%%%%%%%%%%
\item \( \|z - \xlm\| > \|z^{\dagger\dagger} - \xlm\| \).\\  
By using (\ref{eq1.7}) and (\ref{eq14}), we have:
\[
\frac{A_{\muc}(\obj{j}(\bar{z}) - \obj{j}(\xlm)) - A_{\muc}(\obj{j}(z^{\dagger\dagger}) - \obj{j}(\xlm))}{\|x - \xlm\| - \|z^{\dagger\dagger} - \xlm\|}
> \frac{\delta}{\|\bar{z}- \xlm\|} \geq \frac{\delta}{\Kconst}.
\]
Thus, for \( 0< \rhoc < \frac{\delta}{\Kconst}\) and  \( 0< \muc < 1\), the following inequality is satisfied:
\begin{equation*}
  \gdf_{j,\muc, \rhoc, \xlm}(z^{\dagger\dagger}) < \gdf_{j,\muc, \rhoc, \xlm}(\bar{z}). 
\end{equation*}
So, in summary, if \( 0< \rhoc < \frac{\delta}{\Kconst}\) and  \( 0< \muc < 1\), then (\ref{eq13}) holds.
\end{enumerate}
Let,
\begin{equation*}
 z' = \arg\min_{z \in L} \gdf_{\muc, \rhoc, \xlm}(z).  
\end{equation*}
If \( 0< \rhoc < \frac{\delta}{\Kconst}\) and  \( 0< \muc < 1\), then by Inequality (\ref{eq13}), we have:
\begin{equation*}
  z' = \arg\min_{z \in L\setminus B} \gdf_{\muc, \rhoc, \xlm}(z). 
\end{equation*}
Hence, by using (\ref{eq10})-(\ref{eq12}), we have:
% 5. \textbf{Properties of \( z' \):}  
% Since \( z' \in L \setminus B \), we have:
\begin{equation*}
    \obj{j}(z') < \obj{j}(z^{\dagger\dagger}) + \delta \leq \obj{j}(\xlm) - \tau.
\end{equation*}
Now, since \( N(z^{\dagger\dagger},\epsilon) \)  is an open set and \( L \setminus B \) is bounded, then by the continuity of \( \gdf_{j, \muc, \rhoc, \xlm} \), there exists \( 0 < \varepsilon_1 < \varepsilon \) such that whenever \( 0 < \varepsilon' \leq \varepsilon_1 \), we have:
\begin{equation*}
N(\varepsilon',z') \subset LE \setminus EQ \quad \text{and} \quad
\gdf_{j,\muc, \rhoc, \xlm}(z') \leq\gdf_{j,\muc, \rhoc, \xlm}(z) \quad \text{for all } z \in N(\varepsilon',z'),  
\end{equation*}
where \( N(\varepsilon',z') \) is an \( \varepsilon' \)-neighborhood of \( z' \).\\
Therefore, \( z' \) is a weak efficient solution of \( \gdf_{\muc, \rhoc, \xlm} \) over \( \Feas \).
\end{proof}

\subsection{Algorithm}
\label{sec_alg}
%\section{Algorithm}
% This section introduces an algorithm grounded in the theoretical foundations established in the preceding sections. This algorithm aims to find global efficient solutions of non-convex problems. However, the solution of a multi-objective optimization problem is not an isolated minimum but a set of efficient solutions. To generate a well-distributed set of efficient solutions, we have adopted the initial point selection technique developed in \cite{ansary2021sqcqp} (Step 1-4 of Algorithm 6.1). The above ideas are summarized in the following algorithm. In this algorithm, $PF$ denotes the Pareto front of $(\mathrm{MOP})$ prior to multi-objective global descent method, while $PFG$ signifies the Pareto front of $(\mathrm{MOP})$ after multi-objective global descent method. Furthermore, $WPF$ represents the weak Pareto front of $(\mathrm{MOP})$ before MOGDM and $WPFG$ refers to the weak Pareto front of $(\mathrm{MOP})$ after MOGDM.\\
This subsection presents an algorithm developed based on the theoretical results established in the preceding sections. The algorithm's objective is to compute globally efficient solutions for non-convex multi-objective optimization problems. Unlike single-objective problems, the solution to a multi-objective problem is not a unique minimizer but a set of efficient (Pareto-optimal) solutions. To ensure a well-distributed approximation of this set, we utilize the initial point generation strategy proposed in~\cite{ansary2021sqcqp}, specifically Steps 1-4 of Algorithm 6.1. These concepts are consolidated into the algorithm described below. In this context, $(\mathrm{PF})$ denotes the Pareto front of a multi-objective optimization problem before applying the multi-objective global descent method, while $(\mathrm{PFG})$ denotes the corresponding Pareto front obtained after applying the global descent algorithm. Similarly, $(\mathrm{WPF})$ and $(\mathrm{WPFG})$ refer to the weak Pareto fronts before and after applying the global descent method, respectively.
%\begin{figure}[H]
%\centering
 % Ensure enough width after scaling
\begin{algorithm}[H]
\caption{({\it Global descent method for multi-objective optimization models})} \label{alg1}
\vspace{0.5mm}
\begin{algorithmic}
\State \textbf{Phase 0: Initialization}
\begin{enumerate}[label=(\alph*)]
 \item Define the auxiliary function \( V_{\muc} \) satisfying conditions (V1) and (V2).
        \item Select a nonempty subset \( \Feas^0 \subset  \Feas\) consisting of \( N \) initial points.
        \item Use Steps~1--4 of Algorithm~6.1 from~\cite{ansary2021sqcqp} to modify \( N \) initial points that satisfy the condition given in inequality~(5.1) of the same reference. The single-objective global descent method is applied for each objective \( j \in \Idx{m} \) to compute the corresponding ideal and nadir vectors.

        \item Choose initial parameters: \( \rhoc_{\text{ini}} > 0 \), \( 0 < \muc_{\text{ini}} < 1 \), \( \rhoc_L > 0 \), \( \rhoc_U > 0 \), \( 0 < \hat{\rhoc} < 1 \), \( 0 < \hat{\muc} < 1 \), \( \epsilon > 0 \), \( \kappa > 0 \), and \( \bar{\alpha}_U > 0 \).
        \item Set $WPF = WPFG = PF = PFG = \emptyset$. 
        \item Set \( \rhoc := \rhoc_{\text{ini}} \), \( \muc := \muc_{\text{ini}} \).
\end{enumerate}
\State \textbf{Phase 1: Local search}
\begin{enumerate}[label=(\alph*)]
    \item Starting at \( z_{\text{cur}} \), use a suitable local optimization method to compute a local weak efficient solution \( z^{\dagger\dagger} \) of \( \fvec \) over \( \Feas \).
    % \item If \( z^{\dagger\dagger} \) improves upon the current solution, set \( \xlm := z^{\dagger\dagger} \).
    \item Update $WPF=WPF \cup \{\fvec(\xlm)\}$.
    %and increment \( N_{\text{iter}} := N_{\text{iter}} + 1 \).
    \item Proceed to Phase 2 (Global Search).
\end{enumerate}
% \algstore{myalg}
% \end{algorithmic}
% \end{algorithm}
% \begin{algorithm}                     
% \begin{algorithmic} [1]                   % enter the algorithmic environment
% \algrestore{myalg}
\State \textbf{Phase 2: Global search}
\begin{enumerate}[label=(\alph*)]
    \item Generate \( m \) initial points \( \{z_{\text{ini}}^{(i)} \in \Feas \setminus N(\xlm,\epsilon)\}_{i=1}^m \); set \( i := 1 \).\label{step_2a}
    \item Set \( z_{\text{cur}} := z_{\text{ini}}^{(i)} \).\label{step_2b}
    \item Update $\rhoc_L$.
    \item If \( \fvec(z_{\text{cur}}) < \fvec(\xlm) \), then go to Phase 1. \label{step_2c}
%%%%%%
    \item If 
    \begin{align*}
    \|\nabla \gdf_{j,\muc, \rhoc, \xlm}(z_{\text{cur}})\| &< \kappa, \quad \text{for at least one } j \in \Idx{m}, \quad \text{or} \\
    (z_{\text{cur}} - \xlm)^T \nabla \gdf_{j,\muc, \rhoc, \xlm}(z_{\text{cur}}) &\geq 0, \quad \text{for at least one } j \in \Idx{m},
    \end{align*}
    then choose \( l \in \mathbb{N} \) such that:
    \[
    \muc_l := \hat{\muc}^l \muc, \quad \rhoc_l := \hat{\rhoc}^l \rhoc,
    \]
    and for all \( j \in \Idx{m} \),
    \[
    \|\nabla \gdf_{j,\muc_l, \rhoc_l, \xlm}(z_{\text{cur}})\| > \kappa, \quad
    (z_{\text{cur}} - \xlm)^T \nabla \gdf_{j,\muc_l, \rhoc_l, \xlm}(z_{\text{cur}}) < 0.
    \]
    Update \( \muc := \muc_l \), \( \rhoc := \rhoc_l \).\label{step_2d}

    %%%%%%
    \item Determine a descent direction \( D \) for \( \gdf_{\muc, \rhoc, \xlm} \) at \( z_{\text{cur}} \). Perform a line search along \( D \) with step size \( \bar{\alpha} \leq \bar{\alpha}_U \) to obtain a new point \( x \) satisfying descent, and set \( z_{\text{cur}} := z \); go to Step 2\ref{step_2c}.
    \item If \( z \) reaches the boundary of \( \Feas \), skip to the next iteration.
    \item Set \( i := i + 1 \). If \( i \leq m \), return to Step 2\ref{step_2b}.
    \item Otherwise, reduce \( \rhoc := \hat{\rhoc} \rhoc \), reset \( \muc := \muc_{\text{ini}} \). If \( \rhoc \geq \rhoc_L \), return to Step 2\ref{step_2a}; else terminate the algorithm and report \( \xlm \) as the best global weak efficient solution found. Update $WPFG=WPFG \cup \{\fvec(\xlm)\}$.
    \item Construct the Pareto front \( PF \) by filtering out dominated solutions from the set \( WPF \).
    \item Construct the Pareto front \( PFG \) by eliminating dominated solutions from the set \( WPFG \).
\end{enumerate}
\end{algorithmic}
\end{algorithm}
%\end{minipage}
%\end{figure}

\section{Numerical illustrations}\label{sec_exe}
This section describes the numerical illustrations of the proposed global descent method for $(\mathrm{MOP})$. The proposed method is verified and compared with the multi-objective SQCQP method in \cite{ansary2021sqcqp}. A MATLAB (R2024b) code is developed for Algorithm \ref{alg1} ($MOGDM$) with the following structured process that helps us to leverage function transformation techniques.\\
%%%%%%%%%%%%%%%%%%%%%%%%%%%%%%%%%%%%%%%%%%%%%%%%%%%%%%%
%{\bf Implementation details}
\begin{itemize}
    \item \textit{Choosing \( V_{\muc} \):} The auxiliary function \( V_{\muc} \), which satisfies conditions \textnormal{(V1)} and \textnormal{(V2)}, is taken as:
    \begin{equation}
        v_{\muc}(z) = \muc \left[ (1 - c) \left( \frac{\muc(1 - c)}{1 - c\muc} \right)^{z/\tau} + c \right],
    \end{equation}
where \( 0 < c < 1 \).
\par In our implementations, we set \( c = 0.5 \) and \( \tau = 1 \).
    \item \textit{Initial point generation:}A total of 200 uniformly distributed random points are generated within the decision space \( [\text{lb}, \text{ub}] \subset \Feas \subset \mathbb{R}^n \) to initialize the optimization process.
    \item \textit{Ideal and nadir estimation:} The payoff table method~\cite{Miettinen1998-to} is employed to estimate the ideal and nadir vectors. For each objective \( \obj{j} \), the single-objective global descent method~\cite{duanli2010global} is applied using \texttt{fmincon} from three initial points: \( \text{lb} \), \( (\text{lb} + \text{ub})/2 \), and \( \text{ub} \). non-dominated solutions are used to approximate \( \fvec^I \) and \( \fvec^N \).
    \item \textit{Spread:} For generating a well-distributed Pareto front, Steps~1--4 of Algorithm~6.1 from~\cite{ansary2021sqcqp} are employed to refine the initial points such that they satisfy the condition specified in inequality~(5.1) of the same reference. Here, the subproblem \( P_{\text{sp}}(z^0) \), as formulated in Section~5 of~\cite{ansary2021sqcqp}, is solved using MATLAB’s \texttt{quadprog} function.
    \item \textit{Parameter Initialization:} 
Throughout the tests, the algorithm parameters are initialized as follows: \( \muc_{\text{ini}} = 0.5 \), \( \rhoc_{\text{ini}} = \delta / (K + 10^{-3}) \), \( \rhoc_U = \delta / K \), \( \rhoc_L = 10^{-5} \), \( \hat{\rhoc} = 0.35 \), \( \hat{\muc} = 0.1 \) and \( \kappa = 10^{-4} \sqrt{n} \). The parameter \( \epsilon \) is set to \( 1 \) if \( \min(\text{ub}_i - \text{lb}_i) > 10 \) and \( 0.1 \) otherwise. We also set \( \mult_U = \epsilon \). Here, \( K = \| \text{ub} - \text{lb} \| \) and \( \delta = 0.01 \). In Step~2\ref{step_2d}, we use \( l = 1 \).

    \item \textit{Local search:} {A penalty-based sequential quadratic constrained quadratic programming (SQCQP)} approach is used via the \texttt{MOSQCQP} routine~\cite{ansary2021sqcqp}, which internally employs \texttt{fmincon} to identify locally weak efficient solutions $\xlm$.
    \item \textit{Global search:}
    \begin{itemize}
        \item A set of \( m= 2n \) candidate points is selected to minimize the function \( \gdf_{\mu, \rhoc, \xlm} \), either by distributing them symmetrically around the reference point \( \xlm \) or by sampling randomly from the feasible region \( \Feas \). To avoid repeated convergence to the same local optimum, a deleted-\( \epsilon \) neighborhood is constructed such that all new candidate points lie outside a ball of radius \( \epsilon \) centered at \( \xlm \), ensuring the search proceeds in unexplored regions of the feasible space.
                % To avoid convergence to the same local minima, a deleted-\( \epsilon \) neighborhood is constructed using \texttt{generate\_deleted\_epsilon\_neighbourhood\_v2} around each local solution \( x^\star \).
        % \item \alert{calculate $\rhoc_L$}

        \item During the global refinement phase, the lower bound \( \rhoc_L \) for the penalty parameter \( \rhoc \) is adaptively adjusted based on curvature and sensitivity characteristics. For each index \( j \in \mathcal{I}_3 \) satisfying \( \obj{j}(z_{\text{cur}}) < \obj{j}(\xlm) \), the update rule for \( \rhoc_L \) is given by:
% \[
% \rhoc_L := \max \left\{ 0,\ \frac{\mathcal{C} \cdot \max_{j \in \Idx{m}} \left(z_{\text{cur}} - \xlm)^\top \nabla \obj{j}(z_{\text{cur}}) \right) \cdot \|z_{\text{cur}} - \xlm\|}{d^\top (z_{\text{cur}} - \xlm)} \right\} + 10^{-5},
% \]

where \( \mathcal{C} = A'_{\muc} \left( \obj{i}(z_{\text{cur}}) - \obj{i}(\xlm) \right) \). If the resulting value is smaller than a predefined threshold \( \rhoc_U \), it is adopted as the new bound; otherwise, the lower bound is reset to the nominal value of \( 10^{-5} \).
        \item Phase~2\ref{step_2c} corresponds to the case where the current iterate \( z_{\text{cur}} \) yields a value of $\fvec$ with \( \fvec(z_{\text{cur}}) < \fvec(\xlm) \). This indicates that \( z_{\text{cur}} \) lies within a more promising region of the objective space compared to the basin associated with \( \xlm \). Consequently, a local search procedure can be initiated from \( z_{\text{cur}} \) to further reduce the objective values and potentially identify a locally efficient solution superior to \( \xlm \).
% This adjustment improves descent detection in regions where the gradient is shallow or ill-scaled.
        \item According to Theorem~\ref{theorem_3.5}, if \( \fvec(z_{\text{cur}}) \geq \fvec(\xlm) \) and the parameter \( \muc \) is sufficiently small, then the point \( z_{\text{cur}} \) does not correspond to a stationary point of the augmented descent function \( \gdf_{\muc, \rhoc, \xlm} \). In this case, the vector \( (z_{\text{cur}} - \xlm) \) constitutes a valid descent direction for \( \gdf_{\muc, \rhoc, \xlm} \) evaluated at \( z_{\text{cur}} \). To ensure progress, Phase~2\ref{step_2d} iteratively reduces the parameters \( \muc \) and \( \rhoc \) by fixed scaling factors \( \hat{\mu} \) and \( \hat{\rhoc} \), respectively, until the following descent conditions are simultaneously satisfied:
\[
\nabla \gdf_{j,\muc, \rhoc, \xlm}(z_{\text{cur}}) \not \to 0 \quad \text{and} \quad (z_{\text{cur}} - \xlm)^T \nabla \gdf_{j,\muc, \rhoc, \xlm}(z_{\text{cur}}) < 0.
\]
   \item Step sizes are computed via backtracking line search using parameters \( \bar{\alpha}_U = 0.1 \), \( \beta = 10^{-4} \), and contraction ratio \( r = 0.5 \), ensuring both descent and feasibility.
    \item The refinement phase terminates when \( \rhoc < \rhoc_L = 10^{-5} \) or when no valid descent direction exists.
    \end{itemize}

    \item \textit{Pareto front construction:} 
    \begin{itemize}
        \item The Pareto front \( PF \) is constructed by filtering out dominated solutions from \( WPF \), the weak Pareto front before $MOGDM$.
        \item Similarly, the refined Pareto front \( PFG \) is constructed from \( WPFG \), the weak Pareto front obtained after global refinement.
    \end{itemize}

        % \item PF and PFG are used to compare solutions obtained before and after global refinement.
\end{itemize}

%%%%%%%%%%%%%%%%%%%%%%%%%%%%%%%%%%%%%%%%%%%%%%%%%%%%%%%%

\paragraph{Experimental validation}
Algorithm \ref{alg1} is evaluated on a diverse suite of benchmark test problems. Here, we considered the problems to be box-constrained problems. The numerical results demonstrate the proposed method’s robustness and computational efficiency, showing consistent performance across varied problem classes compared to established solvers.

% A MATLAB (R2024b) implementation was developed specifically for this purpose.

\begin{table}[!htbp]
\centering
\scalebox{0.83}{
\tiny  % Slightly bigger than \footnotesize
\setlength{\tabcolsep}{6pt}  % Increase column separation
\renewcommand{\arraystretch}{1.2}  % More vertical padding
%\resizebox{0.95\textwidth}{!}{%
\begin{tabular}{|c|l|c|c|c|c|l|c|c|c|}
\hline
\textbf{Sl. no.} & \textbf{Test problem} & \textbf{(m,n)} & \textbf{MOSQCQP} & \textbf{MOGDM} & 
\textbf{Sl. no.} & \textbf{Test problem} & \textbf{(m,n)} & \textbf{MOSQCQP} & \textbf{MOGDM} \\
\hline
1	&	AL1	&	(2,20)	&	100	&	100	&	23	&	LP1	&	(2,50)	&	48	&	48	\\
2	&	AL2	&	(2,50)	&	56	&	56	&	24	&	LR1	&	(2,50)	&	109	&	106	\\
3	&	CEC09\_1	&	(2,30)	&	196	&	196	&	25	&	GE5	&	(3,3)	&	200	&	200	\\
4	&	CEC09\_2	&	(2,30)	&	193	&	193	&	26	&	IKK1	&	(3,2)	&	200	&	200	\\
5	&	CEC09\_3	&	(2,30)	&	76	&	76	&	27	&	IM1	&	(2,2)	&	183	&	183	\\
6	&	CEC09\_7	&	(2,30)	&	146	&	146	&	28	&	Jin3	&	(2,2)	&	200	&	200	\\
7	&	CEC09\_8	&	(3,30)	&	141	&	141	&	29	&	Jin4\_a	&	(2,2)	&	168	&	175	\\
8	&	Deb513	&	(2,2)	&	17	&	177	&	30	&	KW2	&	(2,2)	&	50	&	51	\\
9	&	Deb521a\_a	&	(2,2)	&	200	&	200	&	31	&	lovison3	&	(2,2)	&	160	&	160	\\
10	&	Deb521b	&	(2,2)	&	200	&	200	&	32	&	lovison4	&	(2,2)	&	183	&	183	\\
11	&	DTLZ1	&	(3,7)	&	147	&	200	&	33	&	lovison5	&	(3,3)	&	89	&	104	\\
12	&	DTLZ1n2	&	(2,2)	&	13	&	200	&	34	&	lovison6	&	(3,3)	&	36	&	63	\\
13	&	DTLZ2	&	(3,12)	&	55	&	55	&	35	&	MOP2	&	(2,4)	&	14	&	170	\\
14	&	DTLZ2n2	&	(2,2)	&	176	&	177	&	36	&	MOP3	&	(2,2)	&	88	&	96	\\
15	&	DTLZ3	&	(3,12)	&	89	&	95	&	37	&	MOP5	&	(3,2)	&	134	&	158	\\
16	&	DTLZ3n2	&	(2,2)	&	16	&	200	&	38	&	MOP6	&	(2,4)	&	17	&	182	\\
17	&	DTLZ5	&	(3,12)	&	105	&	105	&	39	&	Shekel	&	(2,2)	&	182	&	183	\\
18	&	DTLZ5n2	&	(2,2)	&	198	&	198	&	40	&	slcdt1	&	(2,2)	&	200	&	200	\\
19	&	DTLZ6	&	(3,22)	&	164	&	174	&	41	&	VFM1	&	(3,2)	&	178	&	185	\\
20	&	DTLZ6n2	&	(2,2)	&	86	&	199	&	42	&	ZDT1	&	(2,30)	&	173	&	173	\\
21	&	EP2	&	(2,2)	&	200	&	200	&	43	&	ZDT2	&	(2,30)	&	200	&	200	\\
22	&	EX005	&	(2,2)	&	168	&	190	&	44	&	ZDT3	&	(2,30)	&	139	&	159	\\\hline
\end{tabular}
}
\caption{Comparison of (\textit{MOSQCQP}) and (\textit{MOGDM}) for various test problems.}
\label{table_1}
\end{table}
Table~\ref{table_1} presents a detailed comparison between the multi-objective global descent method $(\mathrm{MOGDM})$ and the established multi-objective sequential quadratic constrained quadratic programming ($\mathrm{MOSQCQP}$~\cite{ansary2021sqcqp}) approach across a diverse set of 44 benchmark test problems, encompassing a broad range of dimensional settings in both objectives~(\(m\)) and variables~(\(n\)).

The results show that MOGDM performs robustly across these test problems. In 23 instances, both MOGDM and MOSQCQP yield identical numbers of nondominated solutions, indicating comparable effectiveness on many standard problem types. In 20 problems, MOGDM identifies more nondominated points, particularly in settings with nonconvex structures, multiple local Pareto fronts, or irregular Pareto front geometries. Only one problem slightly favors MOSQCQP, with a negligible difference in outcome.

MOGDM displays especially strong performance on difficult instances such as \texttt{Deb513} (17 vs.\ 177), \texttt{DTLZ1n2} (13 vs.\ 200), \texttt{DTLZ3n2} (16 vs.\ 200), and \texttt{MOP2} (14 vs.\ 170), recovering significantly larger portions of the Pareto front. These examples highlight its capacity to navigate complex landscapes and avoid premature convergence.

Moreover, MOGDM achieves the maximum of 200 nondominated points in 13 problems, including \texttt{Deb521a\_a}, \texttt{Deb521b}, \texttt{DTLZ1n2}, \texttt{DTLZ3n2}, \texttt{GE5}, \texttt{IKK1}, \texttt{Jin3},\linebreak
\texttt{slcdt1}, \texttt{EP2}, \texttt{DTLZ5n2}, \texttt{ZDT2}, \texttt{ZDT3}, and \texttt{MOP6}. In these instances, MOGDM either matches or exceeds the performance of MOSQCQP while offering broader coverage of the Pareto front.

These findings suggest that MOGDM offers a reliable and adaptable framework for multi-objective optimization. It performs consistently well on relatively simple problems and shows clear advantages on more challenging ones, making it a strong candidate for applications requiring high-quality trade-off solutions.
\paragraph{Performance evaluation}
In this part, Algorithm~\ref{alg1} (MOGDM) is benchmarked against the MOSQCQP method \cite{ansary2021sqcqp} and the evolutionary algorithm NSGA-II \cite{deb2002fast} using performance profiles constructed from different quality indicators. The comparison is conducted across the benchmark problems summarized in Table~\ref{table_1}. These profiles are generated based on three widely used metrics in multi-objective optimization: the hypervolume metric, the $\Delta$-spread metric, and the number of function evaluations. For a detailed discussion on performance profiles and metric-based benchmarking, readers may refer to \cite{ansary2020sequential,ansary2021sqcqp, fliege2016sqp}.
\begin{figure}[H]
    \centering
    \begin{subfigure}[b]{0.45\textwidth}
        \centering
        \includegraphics[width=\textwidth,height=2.5cm]{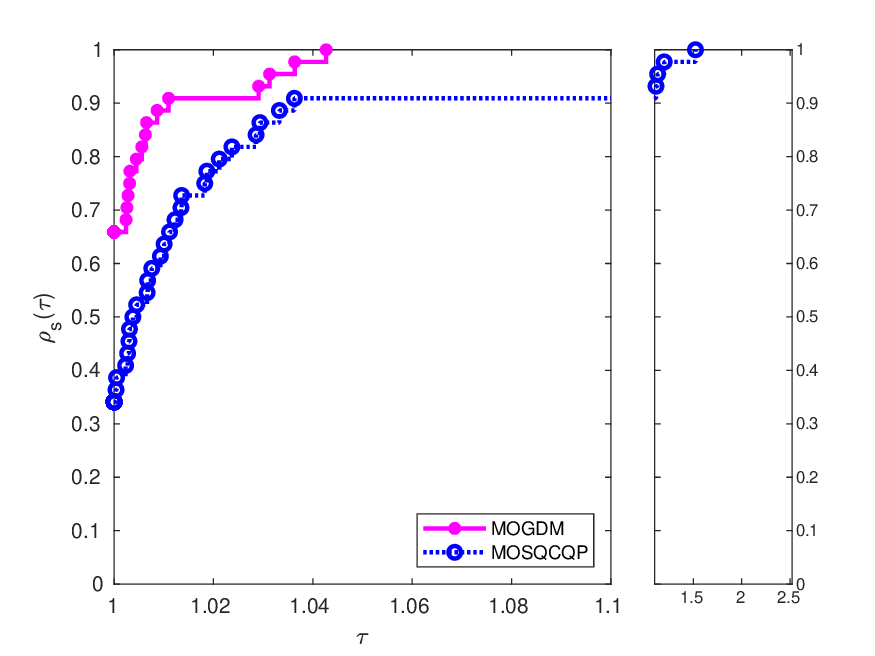}
        \caption{\raggedright Performance profile between MOGDM and MOSQCQP}
        \label{fig_hv_mogdm_mosqcqp}
    \end{subfigure}
    \hfill
    \begin{subfigure}[b]{0.45\textwidth}
        \centering
        \includegraphics[width=\textwidth,height=2.5cm]{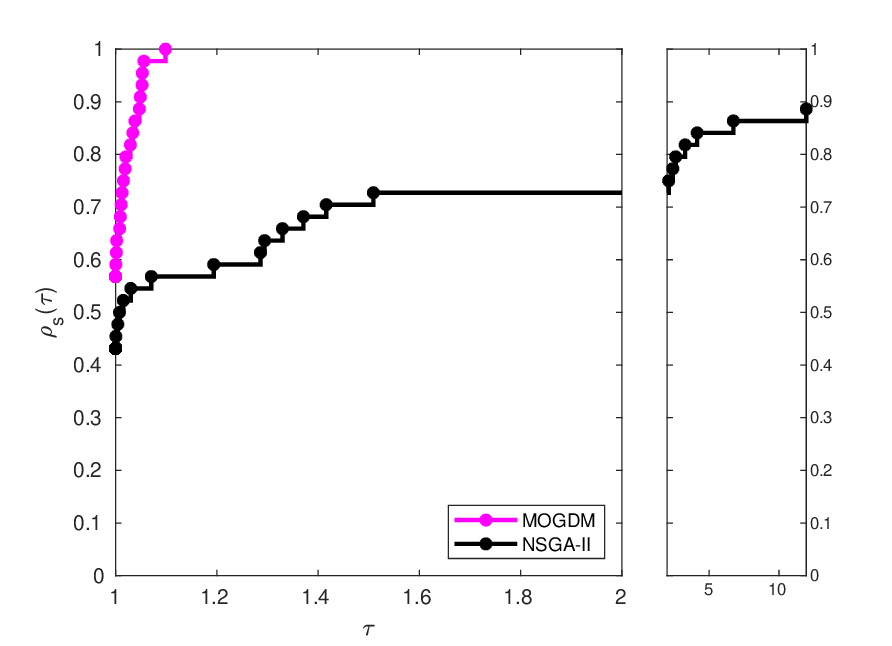}
        \caption{\raggedright Performance profile between MOGDM and NSGA-II}
        \label{fig_hv_mogdm_nsga2}
    \end{subfigure}    
    \caption{Performance profiles using hypervolume metric}
    \label{fig_hv}
\end{figure}

%%%%%%%%%%%%%%%%%%%%%%%%%%%%%%%%% fig2.
\begin{figure}[H]
    \centering
    \begin{subfigure}[b]{0.45\textwidth}
        \centering
        \includegraphics[width=\textwidth,height=2.5cm]{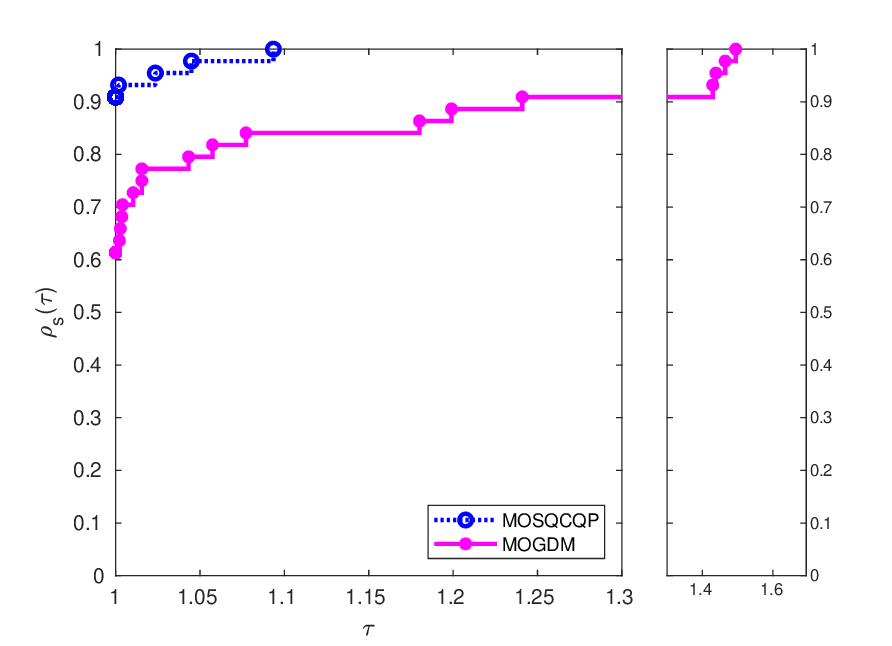}
        \caption{\raggedright Performance profile between MOGDM and MOSQCQP}
        \label{fig_delta_mogdm_mosqcp}
    \end{subfigure}
    \hfill
    \begin{subfigure}[b]{0.45\textwidth}
        \centering
        \includegraphics[width=\textwidth,height=2.5cm]{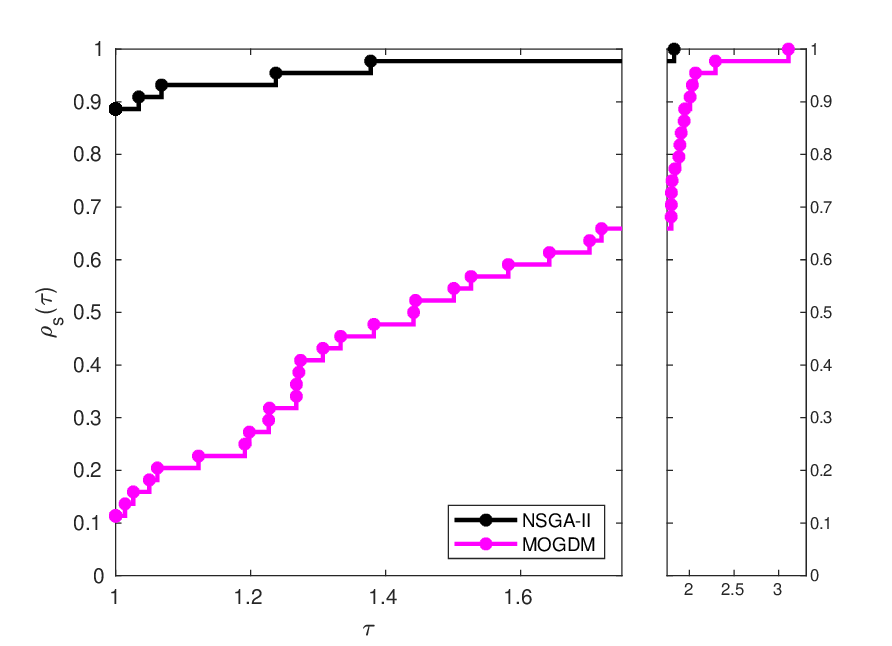}
        \caption{\raggedright Performance profile between MOGDM and NSGA-II}
        \label{fig_delta_mogdm_nsga2}
    \end{subfigure}    
    \caption{Performance profiles using $\Delta$ metric}
    \label{fig_delta}
\end{figure}

%%%%%%%%%%%%%%%%%%%%%%%%%%%%%%%%% fig3.
\begin{figure}[H]
    \centering
    \begin{subfigure}[b]{0.45\textwidth}
        \centering
        \includegraphics[width=\textwidth,height=2.5cm]{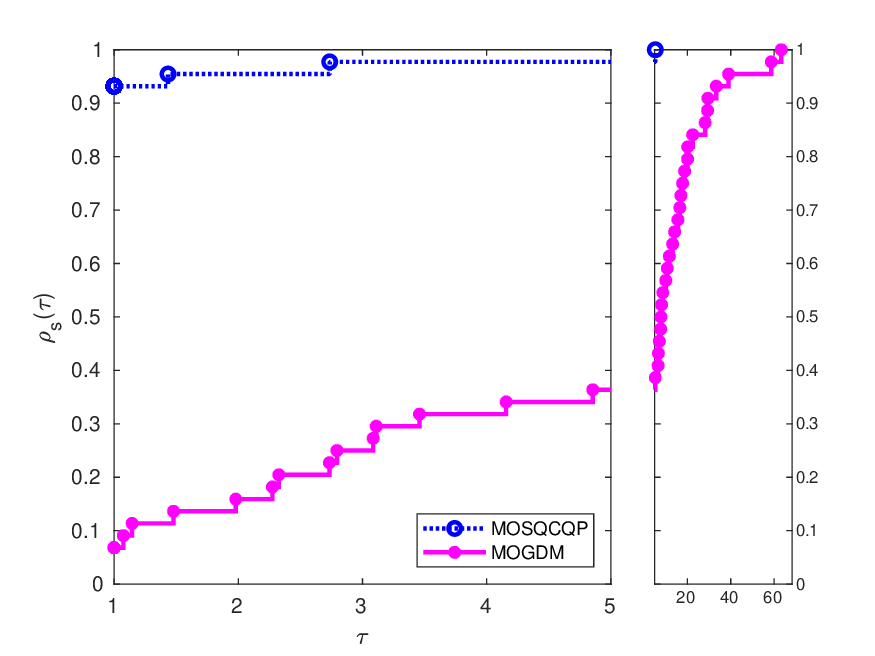}
        \caption{\raggedright Performance profile between MOGDM and MOSQCQP}
        \label{fig_feval_mogdm_mosqcqp}
    \end{subfigure}
    \hfill
    \begin{subfigure}[b]{0.45\textwidth}
        \centering
        \includegraphics[width=\textwidth,height=2.5cm]{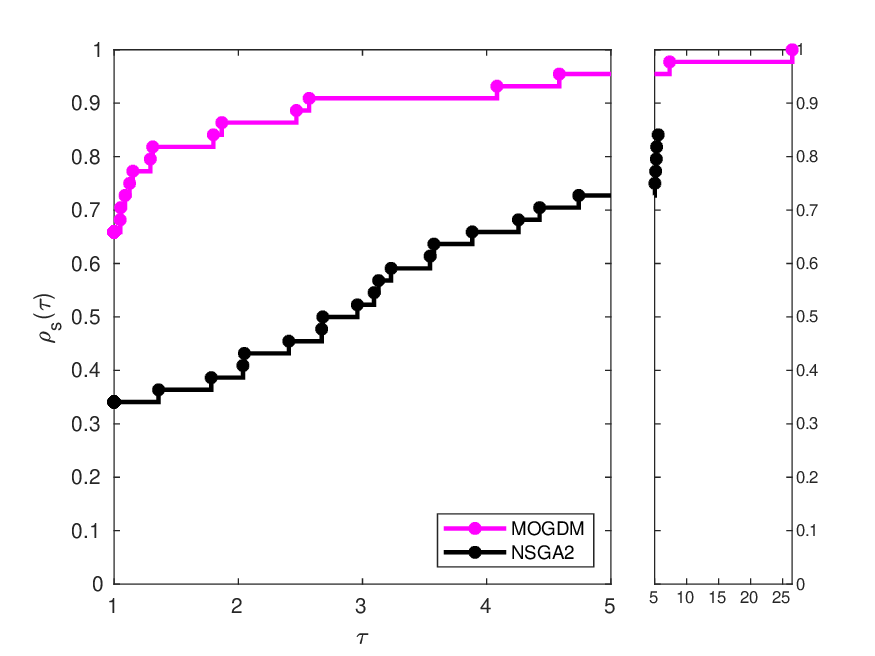}
        \caption{\raggedright Performance profile between MOGDM and NSGA-II}
        \label{fig_feval_mogdm_nsga2}
    \end{subfigure}    
    \caption{Performance profiles using number of function evaluations}
    \label{fig_feval}
\end{figure}
\par Figure~\ref{fig_hv} shows that MOGDM performs consistently better than both MOSQCQP and NSGA-II when evaluated using the hypervolume metric. Its profile dominates across nearly all thresholds, indicating that MOGDM tends to produce solution sets that cover larger portions of the Pareto front. This reflects stronger convergence and diversity properties in terms of objective-space coverage.
\par In Figure~\ref{fig_delta}, which reports performance using the $\Delta$-spread metric, MOSQCQP initially has an advantage over MOGDM, especially at strict thresholds. This suggests that MOSQCQP produces more uniformly spaced solutions in several test problems. However, as the tolerance level increases, the profiles of both methods converge, and MOGDM begins to catch up, indicating competitive average spacing quality.
\par Figure~\ref{fig_feval} presents a comparison based on the number of function evaluations. Since MOSQCQP forms a core part of MOGDM’s Phase 1 (local search), it is expected that MOGDM incurs a higher evaluation count than MOSQCQP. However, MOGDM still clearly outperforms NSGA-II, achieving better solution quality with fewer evaluations. This demonstrates that despite its hybrid nature, MOGDM remains computationally efficient and less expensive than evolutionary approaches like NSGA-II.

Overall, the performance profile analysis shows that MOGDM delivers strong hypervolume performance and efficient convergence while remaining competitive in solution spread. These characteristics make it a robust and well-balanced choice for solving multi-objective optimization problems.

\section{Conclusions and further developments} In this paper, a global descent method has been developed for unconstrained multi-objective optimization problems. The proposed method generated global efficient solutions for non-convex problems with the help of multi-objective global descent functions. In addition to this, the proposed method is free from any kind of a priori chosen parameter or ordering information of the objective function. Advantages of this method over the gradient descent method has been justified with various numerical examples. However, the proposed method is restricted to smooth unconstrained multi-objective optimization problems only. In the near future, we want to extend the ideas of this method to constrained problems as well as nonsmooth problems.

\appendix
\section{Details of test problems}\label{Test_Problems}
Several test problems in Table~\ref{table_1} are adapted from single-objective functions described in the Appendix of~\cite{duanli2010global}, and are extended in this work to a multi-objective optimization setting.
\begin{enumerate}
\item AL1 ($m=2$, $n=20$)
\begin{align*}
\min_{z \in \mathbb{R}^n} \quad & \left( \obj{1}(z), \obj{2}(z) \right) \\
\text{s.t.} \quad & -0.5 \leq z_i \leq 1.5, \quad \forall i \in \Idx{n} \\
\text{where} \quad 
& \obj{1}(z) = c_1 + e - c_1 \exp\left( -c_2 \sqrt{ \frac{1}{n} \sum_{i=1}^n z_i^2 } \right) - \exp\left\{ \frac{1}{n} \sum_{i=1}^n \cos(c_3 z_i) \right\}, \\
& \obj{2}(z) = \frac{\pi}{n} \Bigg[
k \sin^2(\pi z_1) 
+ \sum_{i=1}^{n-1} (z_i - a)^2 \left\{1 + k \sin^2(\pi z_{i+1}) \right\} \\
& \quad \quad \quad + (z_n - a)^2 \Bigg], \\
& \text{with parameters } 
n = 20,\; c_1 = 20,\; c_2 = 0.2,\; c_3 = 2\pi,\; k = 10,\; a = 1.
\end{align*}

\item AL2 ($m=2$, $n=50$)
\begin{align*}
\min_{z \in \mathbb{R}^n} \quad & \left(  \obj{1}(z),\;  \obj{2}(z) \right) \\
\text{s.t.} \quad & -0.5 \leq z_i \leq 1.5, \quad \forall i \in \Idx{n} \\
\text{where} \quad 
& \obj{1}(z) = c_1 + \mathrm{e} - c_1 \exp\left( -c_2 \sqrt{ \frac{1}{n} \sum_{i=1}^n z_i^2 } \right) 
               - \exp \Bigg\{ \frac{1}{n} \sum_{i=1}^n \cos(c_3 z_i) \Bigg\}, \\
& \obj{2}(z) = k_1 \Bigg[ \sin^2(\pi \ell_0 z_1) 
            + \sum_{i=1}^{n-1} (z_i - a)^2 \{1 + k_0 \sin^2(\pi \ell_1 z_{i+1}) \} \\
& \quad\quad\quad\quad + (z_n - a)^2 \{1 + k_0 \sin^2(2\pi \ell_1 z_n) \} \Bigg], \\
& \text{with parameters: } n = 50,\; \ell_0 = 3,\; \ell_1 = 2,\; c_1 = 20,\; c_2 = 0.2,\;\\
& \hspace{2.8cm} c_3 = 2\pi,\; k_0 = 1,\; k_1 = 0.1,\; a = 1.
\end{align*}

\item LP1 ($m=2$, $n=50$)
\begin{align*}
\min_{z \in \mathbb{R}^n} \quad & \left( \obj{1}(z),\;  \obj{2}(z) \right) \\
\text{s.t.} \quad & -3 \leq z_i \leq 2, \quad \forall i \in \Idx{n} \\
\text{where} \quad 
& \obj{1}(z) = \frac{\pi}{n} \Bigg[
k \sin^2(\pi z_1) 
+ \sum_{i=1}^{n-1}  (z_i - a)^2 \{1 + k \sin^2(\pi z_{i+1}) \}  \\
& \quad \quad \quad + (z_n - a)^2
\Bigg], \\
& \obj{2}(z) = \frac{1}{2} \sum_{i=1}^{n} \left(z_i^4 - 16z_i^2 + 5z_i\right), \\
& \text{with parameters: } n = 50,\;  k= 10,\; a = 1.
\end{align*}
%%%

\item LR1 ($m=2$, $n=50$)
\begin{align*}
\min_{z \in \mathbb{R}^n} \quad & \left( \obj{1}(z),\; \obj{2}(z) \right) \\
\text{s.t.} \quad & -2 \leq z_i \leq 2, \quad \forall i \in \Idx{n} \\
\text{where} \quad 
& \obj{1}(z) =\frac{\pi}{n} \Bigg[
k \sin^2(\pi z_1) 
+ \sum_{i=1}^{n-1}  (z_i - a)^2 \{1 + k \sin^2(\pi z_{i+1}) \}  \\
& \quad \quad \quad + (z_n - a)^2
\Bigg], \\
& \obj{2}(z) = A n + \sum_{i=1}^{n} \left( z_i^2 - A \cos(\omega z_i) \right), \\
& \text{with parameters: } n = 50,\; k = 10,\; a = 1,\; A = 10,\; \omega = 2\pi.
\end{align*}

% \item LR2 ($m=2$, $n=50$)
% \begin{align*}
% \min_{z \in \mathbb{R}^n} \quad & \left( \obj{1}(z),\; \obj{2}(z) \right) \\
% \text{s.t.} \quad & -5 \leq z_i \leq 5, \quad \forall i \in \Idx{n} \\
% \text{where} \quad 
% & \obj{1}(z) = k_1 \Bigg[ \sin^2(\pi \ell_0 z_1) 
% + \sum_{i=1}^{n-1} (z_i - a)^2 \left( 1 + k_0 \sin^2(\pi \ell_1 z_{i+1}) \right) \\
% & \quad\quad\quad\quad\quad + (z_n - a)^2 \left( 1 + k_0 \sin^2(\pi \ell_1 z_n) \right) \Bigg], \\
% & \obj{2}(z) = A n + \sum_{i \in \Idx{n}} \left( z_i^2 - A \cos(\omega z_i) \right), \\
% & \text{with parameters: } n = 50,\; k_0 = 1,\; k_1 = 0.1,\; a = 1,\; \ell_0 = 3,\; \ell_1 = 2,\;\\
% & \hspace{2.8cm} A = 10,\; \omega = 2\pi.
% \end{align*}

\item 
For detailed formulations and additional information on the remaining test problems listed in Table~\ref{table_1}, the reader is referred to the works~\cite{Evtushenko2013-ii, fliege2016sqp, fonseca1995overview,schutze2008convergence,Zilinskas2014-po}.
\end{enumerate}
$ $\\
\section*{Declarations}
% Some journals require declarations to be submitted in a standardised format. Please check the Instructions for Authors of the journal to which you are submitting to see if you need to complete this section. If yes, your manuscript must contain the following sections under the heading `Declarations':
\begin{itemize}
\item {\it Funding:} This article is partially funded by {\it Prime Minister's Research Fellows (PMRF) scheme}. Author Bikram Adhikary acknowledges funding from the PMRF scheme (PMRF ID. 2202747), India.
\item Competing interests:  The authors declare that they have no competing interests.
\item Ethics approval and consent to participate: Not applicable.
\item Consent for publication: Not applicable.
\item Data availability: In this paper, we have not used any associated data. 
\item Materials availability: The code used for simulations is available upon reasonable request from the corresponding author.
\item Code availability: Code can be made available upon reasonable request.
\item Author contribution: All authors have contributed equally in theoretical developments as well as numerical calculations.
\end{itemize}
%%===========================================================================================%%
\bibliographystyle{sn-mathphys-num}
%\bibliography{MOGD}% common bib file

\begin{thebibliography}{10}

\bibitem{ansary2023proximal}
{\sc M.~A.~T. Ansary}, {\em A {Newton-type} proximal gradient method for nonlinear multi-objective optimization problems}, Optim. Methods Softw., 38 (2023), pp.~570--590.

\bibitem{ansary2015modified}
{\sc M.~A.~T. Ansary and G.~Panda}, {\em A modified {quasi-Newton} method for vector optimization problem}, Optimization, 64 (2015), pp.~2289--2306.

\bibitem{ansary2019sequential}
{\sc M.~A.~T. Ansary and G.~Panda}, {\em A sequential quadratically constrained quadratic programming technique for a multi-objective optimization problem}, Eng. Optim., 51 (2019), pp.~22--41.

\bibitem{ansary2020sequential}
{\sc M.~A.~T. Ansary and G.~Panda}, {\em A sequential quadratic programming method for constrained multi-objective optimization problems}, J. Appl. Math. Comput., 64 (2020), pp.~379--397.

\bibitem{ansary2021sqcqp}
{\sc M.~A.~T. Ansary and G.~Panda}, {\em A globally convergent {SQCQP} method for multiobjective optimization problems}, SIAM J. Optim., 31 (2021), pp.~91--113.

\bibitem{bento2014proximal}
{\sc G.~C. Bento, J.~X. Cruz~Neto, and A.~Soubeyran}, {\em A proximal point-type method for multicriteria optimization}, Set-Valued Var. Anal., 22 (2014), pp.~557--573.

\bibitem{Deb2001-tt}
{\sc K.~Deb}, {\em {Multi-Objective} Optimization using Evolutionary Algorithms}, Wiley Interscience Series in Systems and Optimization, John Wiley \& Sons, Chichester, England, May 2001.

\bibitem{deb2002fast}
{\sc K.~Deb, A.~Pratap, S.~Agarwal, and T.~Meyarivan}, {\em A fast and elitist multiobjective genetic algorithm: Nsga-ii}, IEEE Trans. Evol. Comput., 6 (2002), pp.~182--197.

\bibitem{Ehrgott2005}
{\sc M.~Ehrgott}, {\em Multicriteria Optimization}, Springer, Berlin, Germany, 2~ed., May 2005.

\bibitem{Evtushenko2013-ii}
{\sc Y.~G. Evtushenko and M.~A. Posypkin}, {\em Nonuniform covering method as applied to multicriteria optimization problems with guaranteed accuracy}, Comput. Math. Math. Phys., 53 (2013), pp.~144--157.

\bibitem{fliege2009newton}
{\sc J.~Fliege, L.~G. Drummond, and B.~F. Svaiter}, {\em Newton's method for multiobjective optimization}, SIAM J. Optim., 20 (2009), pp.~602--626.

\bibitem{fliege2000steepest}
{\sc J.~Fliege and B.~F. Svaiter}, {\em Steepest descent methods for multicriteria optimization}, Math. Methods Oper. Res, 51 (2000), pp.~479--494.

\bibitem{fliege2016sqp}
{\sc J.~Fliege and A.~I.~F. Vaz}, {\em A method for constrained multiobjective optimization based on {SQP} techniques}, SIAM J. Optim., 26 (2016), pp.~2091--2119.

\bibitem{fonseca1995overview}
{\sc C.~M. Fonseca and P.~J. Fleming}, {\em An overview of evolutionary algorithms in multiobjective optimization}, Evol. Comput., 3 (1995), pp.~1--16.

\bibitem{kumar1}
{\sc S.~Kumar, M.~A.~T. Ansary, N.~K. Mahato, and D.~Ghosh}, {\em Steepest descent method for uncertain multiobjective optimization problems under finite uncertainty set}, Appl. Anal.,  (2024), pp.~1--22.

\bibitem{kumar2}
{\sc S.~Kumar, M.~A.~T. Ansary, N.~K. Mahato, D.~Ghosh, and Y.~Shehu}, {\em Newton's method for uncertain multiobjective optimization problems under finite uncertainty sets.}, J. Nonlinear Var. Anal., 7 (2023).

\bibitem{kumar3}
{\sc S.~Kumar, N.~K. Mahato, M.~A.~T. Ansary, D.~Ghosh, and S.~Trean{\c{t}}{\u{a}}}, {\em A modified quasi-newton method for uncertain multiobjective optimization problems under a finite uncertainty set}, Eng. Optim.,  (2024), pp.~1--30.

\bibitem{ram1}
{\sc K.~K. Lai, S.~K. Mishra, G.~Panda, M.~A.~T. Ansary, and B.~Ram}, {\em On q-steepest descent method for unconstrained multiobjective optimization problems}, AIMS Math., 5 (2020), pp.~5521--5540.

\bibitem{laumanns2002combining}
{\sc M.~Laumanns, L.~Thiele, K.~Deb, and E.~Zitzler}, {\em Combining convergence and diversity in evolutionary multiobjective optimization}, Evol. Comput., 10 (2002), pp.~263--282.

\bibitem{Miettinen1998-to}
{\sc K.~Miettinen}, {\em Nonlinear Multiobjective Optimization}, International Series in Operations Research \& Management Science, Springer, Dordrecht, Netherlands, 1998~ed., Sept. 1998.

\bibitem{Miettinen2012-rv}
{\sc K.~Miettinen}, {\em Nonlinear Multiobjective Optimization}, International series in operations research \& management science, Springer, New York, Oct 2012.

\bibitem{ram2}
{\sc S.~K. Mishra, G.~Panda, M.~A.~T. Ansary, and B.~Ram}, {\em On {q-Newton’s} method for unconstrained multiobjective optimization problems}, J. Appl. Math. Comput., 63 (2020), pp.~391--410.

\bibitem{mostaghim2007multi}
{\sc S.~Mostaghim, J.~Branke, and H.~Schmeck}, {\em Multi-objective particle swarm optimization on computer grids}, in Proceedings of the $9^{th}$ annual conference on Genetic and evolutionary computation, 2007, pp.~869--875.

\bibitem{duanli2010global}
{\sc C.~K. Ng, D.~Li, and L.~S. Zhang}, {\em Global descent method for global optimization}, SIAM J. Optim., 20 (2010), pp.~3161--3184.

\bibitem{qu2013trust}
{\sc S.~Qu, M.~Goh, and B.~Liang}, {\em {Trust} region methods for solving multiobjective optimisation}, Optim. Methods Softw., 28 (2012), pp.~796--811.

\bibitem{renpu1990filled}
{\sc G.~Renpu}, {\em A filled function method for finding a global minimizer of a function of several variables}, Math. Program., 46 (1990), pp.~191--204.

\bibitem{schutze2008convergence}
{\sc O.~Sch{\"u}tze, M.~Laumanns, C.~A. Coello~Coello, M.~Dellnitz, and E.-G. Talbi}, {\em Convergence of stochastic search algorithms to finite size pareto set approximations}, J. Glob. Optim., 41 (2008), pp.~559--577.

\bibitem{tanabe2019proximal}
{\sc H.~Tanabe, E.~H. Fukuda, and N.~Yamashita}, {\em Proximal gradient methods for multiobjective optimization and their applications}, Comput. Optim. Appl., 72 (2019), pp.~339--361.

\bibitem{zhang2004new}
{\sc L.-S. Zhang, C.-K. Ng, D.~Li, and W.-W. Tian}, {\em A new filled function method for global optimization}, J. Glob. Optim., 28 (2004), pp.~17--43.

\bibitem{Zilinskas2014-po}
{\sc A.~{\v Z}ilinskas}, {\em A statistical model-based algorithm for `black-box' multi-objective optimisation}, Int. J. Syst. Sci., 45 (2014), pp.~82--93.

\end{thebibliography}

\end{document}